\documentclass[12pt,a4paper.,reqno]{amsart}
\pdfoutput=1
\usepackage{amsfonts}
\usepackage{amsthm}
\usepackage{amsmath}
\usepackage{amssymb}
\usepackage{mathabx}
\usepackage{bm}
\usepackage{amscd}
\usepackage[latin2]{inputenc}
\usepackage{t1enc}
\usepackage[mathscr]{eucal}
\usepackage{indentfirst}
\usepackage{graphicx}
\usepackage{graphics}
\usepackage{pict2e}
\usepackage{epic}
\numberwithin{equation}{section}
\usepackage[margin=2.9cm]{geometry}
\usepackage{epstopdf} 
\usepackage[colorlinks,linkcolor=blue]{hyperref}
\usepackage{todonotes}

\usepackage[capitalise,noabbrev]{cleveref}

\crefformat{equation}{(#2#1#3)}
\crefmultiformat{equation}{(#2#1#3)}{ and~(#2#1#3)}{, (#2#1#3)}{ and~(#2#1#3)}
\crefrangeformat{equation}{(#3#1#4) to~(#5#2#6)}

\usepackage[normalem]{ulem}

\allowdisplaybreaks

\theoremstyle{plain}
\newtheorem{Thm}{Theorem}[section]
\newtheorem*{Thm*}{Theorem}
\newtheorem{Lem}[Thm]{Lemma}

\newtheorem{Prop}[Thm]{Proposition}

\theoremstyle{definition}

\newtheorem{Rem}[Thm]{Remark}
\newtheorem{?}[Thm]{Problem}

\newcommand{\ovl}{\overline}
\newcommand{\p}{\partial}
\newcommand{\R}{\mathbb{R}}
\newcommand{\e}{\varepsilon}

\newcommand{\E}{\mathbf{e}}
\newcommand{\h}{\mathbf{h}}
\newcommand{\g}{\mathbf{g}}
\newcommand{\f}{\mathbf{f}}

\newcommand{\Torus}{\mathbb{T}}

\newcommand{\dv}{\text{div}}
\newcommand{\lap}{\triangle}
\newcommand{\dnab}{\cdot \nabla}

\newcommand{\rhob}{\ovl{\rho}}
\newcommand{\ub}{\ovl{u}}
\newcommand{\rhot}{\tilde{\rho}}
\newcommand{\ut}{\tilde{u}}
\newcommand{\uv}{\mathbf{u}}
\newcommand{\wv}{\mathbf{w}}
\newcommand{\zv}{\mathbf{z}}
\newcommand{\uvb}{\ovl{\mathbf{u}}}
\newcommand{\uvt}{\tilde{\mathbf{u}}}
\newcommand{\ot}{\tilde{\omega}}

\newcommand{\abs}[1]{\left\lvert#1\right\rvert}

\newcommand{\norm}[1]{\left\lVert#1\right\rVert}

\let\oldtocsection=\tocsection

\let\oldtocsubsection=\tocsubsection

\let\oldtocsubsubsection=\tocsubsubsection

\renewcommand{\tocsection}[2]{\hspace{0em}\oldtocsection{#1}{#2}}
\renewcommand{\tocsubsection}[2]{\hspace{1em}\oldtocsubsection{#1}{#2}}
\renewcommand{\tocsubsubsection}[2]{\hspace{2em}\oldtocsubsubsection{#1}{#2}}

\begin{document}


\title[rarefaction waves under periodic perturbations]{Asymptotic stability of planar rarefaction waves for 3-d isentropic Navier-Stokes equations under periodic perturbations}

\author[F. Huang]{Feimin HUANG$ ^{1,2} $}

\address{$^1$ Academy of Mathematics and Systems Science,  CAS, Beijing 100190, China}
\address{$^2$ School of Mathematical Sciences, University of Chinese Academy of Sciences, Beijing 100049, China}
\address{$^3$ Department of Mathematics, Yau Mathematical Sciences Center, Tsinghua University, Beijing 100084, China}

\thanks{$ * $ Corresponding author}
\thanks{Feimin Huang is partially supported by NSFC Grant No. 11371349 and 11688101.}

\email{fhuang@amt.ac.cn}

\author[L. Xu]{Lingda Xu$ ^3 $}
\email{xulingda@tsinghua.edu.cn}

\author[Q. Yuan]{Qian YUAN$ ^{1,*} $}
\thanks{Qian Yuan is supported by the China Postdoctoral Science Foundation funded projects 2019M660831 and 2020TQ0345. }
\email{qyuan103@link.cuhk.edu.hk}

\maketitle

\begin{abstract}
We study the asymptotic stability of a planar rarefaction wave (in the $ x_1 $- direction) for the 3-d isentropic Navier-Stokes equations, where the initial perturbation is periodic  on the torus $ \Torus^3 $ with zero average.
To solve this Cauchy problem in which the initial data is periodic with respect to only $ x_2 $ and $ x_3 $ but not to $ x_1, $ we construct a suitable ansatz carrying the oscillations of the solution in the $ x_1 $- direction, but remaining to be periodic in the transverse $ x_2 $- and $ x_3 $- directions. 
In such a way, the difference between the ansatz and the solution can be integrable on the region $ \R\times\mathbb{T}^2, $ which allows us to utilize the energy method with the aid of a Gagliardo-Nirenberg type inequality on $ \R\times\mathbb{T}^2 $ to prove the result.
\end{abstract}

\tableofcontents


\section{Introduction}

We consider a Cauchy problem of the three-dimensional (3-d) isentropic compressible Navier-Stokes (CNS) equations, which read in $ \R^3 $ as
\begin{equation}\label{NS}
\begin{cases}
\p_t \rho + \dv \left( \rho \uv \right) = 0, \\
\p_t \left(\rho \uv \right) + \dv \left( \rho \uv \otimes \uv \right) + \nabla p(\rho) = \mu \lap \uv + \left(\mu+\lambda \right) \nabla \dv \uv,
\end{cases}
\end{equation}
where $ t>0, x=(x_1,x_2,x_3)\in \R^3, $ $ \rho(x,t)>0 $ is the density, $ \uv(x,t) = \left( u_1,u_2,u_3 \right)(x,t)\in\R^3 $ is the velocity, the pressure $ p(\rho) $ satisfies $ p(\rho) = \rho^\gamma $ with $ \gamma\geq 1, $ and $ \mu > 0 $ and $ 
\lambda+\frac{2}{3}\mu \geq 0 $ are the viscous coefficients.

When $ \mu = \lambda =0, $ \cref{NS} turns to the 3-d isentropic compressible Euler equations. A planar centered rarefaction wave $ \left(\rho^r, \uv^r\right)(x,t) = \left(\rho^r, u_1^r, 0,0\right)(x_1 ,t) $ is a weak entropy solution to this hyperbolic system, where $ \left(\rho^r, u_1^r \right) $ solves the following Riemann problem in one dimension,
\begin{equation}\label{Euler}
\begin{cases}
\p_t \rho + \p_1  \left( \rho u_1 \right) = 0, & \\
\p_t \left(\rho u_1 \right) + \p_1  \left( \rho u_1^2 \right) + \p_1  p(\rho) = 0, & \\
\left( \rho, u_1 \right)(x_1 ,0) = \begin{cases}
\left( \rhob^-, \ub_1^- \right), \quad x_1 <0, & \\
\left( \rhob^+, \ub_1^+ \right), \quad x_1 >0. &
\end{cases}
\end{cases}
\end{equation}
In this paper, we consider only the 2-rarefaction wave, i.e., the constants of the initial data in \cref{Euler} satisfy the relation
\begin{equation}\label{rel-const}
\ub_1^+ = \ub_1^- + \int_{\rhob^-}^{\rhob^+} \frac{\sqrt{p'(s)}}{s} ds \quad \text{ with } \quad \rhob^- < \rhob^+.
\end{equation}
We remark that the cases for the 1-rarefaction wave and a combination of two families of rarefaction waves can be proved in a similar way.

The 2-rarefaction wave $ \left(\rho^r, u_1^r\right)(x_1,t) $ can be solved as follows.
For $ \rho>0, $ the system \cref{Euler} is strictly hyperbolic with two distinct eigenvalues $$ \lambda_1(\rho,u_1) = u_1-\sqrt{p'(\rho)},\quad \lambda_2(\rho,u_1) = u_1 +\sqrt{p'(\rho)}. $$ 
One can normalize the corresponding right eigenvectors $ r_i $  as $ \nabla \lambda_i \cdot r_i \equiv 1 ~ (i=1,2). $ And the i-Riemann invariant $ (i=1,2) $ is given by 
$$ Z_i(\rho,u_1) = u_1 + (-1)^{i+1} \int_{1}^{\rho} \frac{\sqrt{p'(s)}}{s} ds, $$ which satisfies $ \nabla Z_i \cdot r_i \equiv 0. $
Denote $ \lambda_2^\pm := \lambda_2\left(\rhob^\pm, \ub_1^\pm \right) $ and $ Z_2^\pm := Z_2\left(\rhob^\pm, \ub_1^\pm \right). $
Then the 2-centered rarefaction wave $ \left(\rho^r,u_1^r\right) $ can be solved exactly by 
\begin{equation}\label{solve-rare}
\begin{aligned}
& \lambda_2\left(\rho^r(x_1 ,t), u_1^r(x_1 ,t)\right) = u_1^r(x_1 ,t)+\sqrt{p'(\rho^r)}(x_1 ,t) = \omega \left(\frac{x_1}{t}\right), \\ 
& Z_2\left(\rho^r(x_1 ,t), u_1^r(x_1 ,t)\right) = u_1^r(x_1 ,t) - \int_{1}^{\rho^r(x_1 ,t)} \frac{\sqrt{p'(s)}}{s} ds = Z_2^- \big(= Z_2^+\big),
\end{aligned}
\end{equation}
where
\begin{equation*}
\omega \left(\xi \right) = \begin{cases}
\lambda_2^-, \quad & \xi < \lambda_2^-, \\
\xi, \quad &  \lambda_2^- \leq \xi < \lambda_2^+, \\
\lambda_2^+, \quad & \xi \geq  \lambda_2^+.
\end{cases}
\end{equation*}

It is well-known that the compressible Euler equations have three important entropy solutions, the shock, the rarefaction wave and the contact discontinuity, which are called the Riemann solutions. For the 1-d case, it has been shown in many literatures that these Riemann solutions characterize the large time behaviors of the solutions, as long as the initial data tend to constant states at far field. In other words, the Riemann solutions are time-asymptotically stable if the initial perturbations are integrable on $ \R $ (at least $ L^1\cap L^2 $- integrable).
For the CNS equations, due to the effect of viscosity, the large time behaviors of the solutions are governed by the viscous versions of these three basic waves, i.e. the viscous shock wave, the rarefaction wave and the viscous contact discontinuity.
For instance, Goodman \cite{G} and Matsumura-Nishihara \cite{MN} independently proved the stability of a single viscous shock wave with a zero-mass condition by using the anti-derivative method. And then there have been a lot of efforts \cite{Liu1,Liu3,SX} to remove the zero-mass condition, where the key is to introduce some diffusion waves propagating along the directions of other characteristic families in order to carry the excessive masses. For the rarefaction wave, Matsumura-Nishihara \cite{MN3} was the first to show the stability for the isentropic CNS system, where the initial perturbations can be large but the density is away from the vacuum, and we also refer to \cite{NYZ} for the stability result of the full CNS system. For the contact discontinuity, Huang-Matsumura-Xin \cite{HMX} and Huang-Xin-Yang \cite{HXY} proved the stability with an algebraic decay rate by constructing a suitable ansatz and using the anti-derivative method.

For the multi-dimensional (m-d) wave patterns,  Xin \cite{Xin1990} showed in 1990 that the planar rarefaction waves are stable for the scalar viscous conservation laws through an $L^2$- energy method. Since then, \cite{Ito1996,NN2000,KNM2004} and the reference therein further improved the results of Xin \cite{Xin1990}, where both the strength of the wave and the initial perturbation can be large and an optimal decay rate $ t^{-\frac{1}{2}} $ is also obtained. 
However, for the case of the m-d Navier-Stokes equations, the stability of the planar rarefaction waves is still a challenging open problem.  
For a 2 $\times$ 2
system with an artificial viscosity matrix, Hokari-Matsumura \cite{HoM} proved the stability of the planar rarefaction wave in two dimensions,
which crucially depends on the strict positivity of the viscosity matrix. Recently, Li-Wang \cite{LW} and Li-Wang-Wang \cite{LWW} showed the stability of the planar rarefaction waves for the Navier-Stokes equations on the domain $ \R\times\Torus $ and on an infinitely long nozzle domain $\mathbb{R}\times\mathbb{T}^2, $ respectively, where the periodic boundary conditions are imposed in their settings.

In this paper, we consider a Cauchy problem of the 3-d isentropic CNS system, concerning the stability of the planar rarefaction waves under space-periodic perturbations. 
It is important and interesting to study the stability of the Riemann solutions under periodic perturbations, where the initial data tend to different periodic functions at far fields. It was shown in \cite{XYY2019,Xin2019,YY2019} that for the 1-d scalar conservation laws, the periodic oscillations around the shocks and rarefaction waves can be canceled as time increases due to the genuine nonlinearity of the flux.
Recently, Huang-Yuan \cite{HY2020} further studied the m-d scalar viscous conservation laws to show the stability of the scalar planar rarefaction waves under m-d periodic perturbations. In particular, a Gagliardo-Nirenberg type inequality on $ \R\times\Torus^{n-1} $ was introduced, which plays an important role when doing energy estimates on this unbounded domain without zero boundary conditions. We also refer to \cite{HY1shock,YY2shock} for the viscous shocks under periodic perturbations for the 1-d CNS equations.

\vspace{0.3cm}

Now we formulate the main result of this paper.
Since the centered rarefaction wave $ \left(\rho^r, u_1^r\right)(x_1,t) $ is only Lipschitz continuous, we need to construct a smooth approximation as in \cite{Matsumura1986}. 
Let $ \ot(x_1 ,t) $ be the unique smooth solution to the problem,
\begin{equation*}
\begin{cases}
\p_t \ot + \p_1  \big( \frac{\ot^2}{2} \big) =0, \\
\ot(x_1 ,0) = \frac{\lambda_2^- + \lambda_2^+}{2} + \frac{\lambda_2^+-\lambda_2^-}{2} \text{tanh}(x_1).
\end{cases}
\end{equation*}
Same as \cref{solve-rare}, we let $ \left(\rhot^r, \ut_1^r\right) $ be the smooth functions solved uniquely by
\begin{equation}\label{def-sm-rare}
\begin{aligned}
\ut_1^r(x_1 ,t)+\sqrt{p'(\rhot^r)}(x_1 ,t) & = \ot(x_1 ,t), \\ 
\ut_1^r(x_1 ,t) - \int_{1}^{\rhot^r(x_1 ,t)} \frac{\sqrt{p'(s)}}{s} ds & = Z_2^- \big(= Z_2^+\big).
\end{aligned}
\end{equation}
From this, one has that
\begin{equation}\label{rel-key}
\ut_1^r(x_1 ,t) = \ub_1^- + \int_{\rhob^-}^{\rhot^r(x_1 ,t)} \frac{\sqrt{p'(s)}}{s} ds.
\end{equation}


To study the stability of the planar rarefaction waves under 3-d periodic perturbations, we prescribe the initial data for \cref{NS} as 
\begin{equation}\label{ic}
\left(\rho, \rho\uv\right)(x,0) =  \left(\rhot^r, \rhot^r \ut_1^r, 0, 0\right)(x_1) + \left(v_0, \wv_0 \right)(x), \quad x\in\R^3,
\end{equation}
where $ v_0(x) \in\R $ and $ \wv_0(x) = \left(w_{1,0},w_{2,0},w_{3,0} \right)(x) \in \R^3 $ are periodic functions defined on the 3-d torus $ \Torus^3 := [0,1]^3, $ satisfying
\begin{equation}\label{zero-ave}
\int_{\Torus^3} \left(v_0, \wv_0\right)(x) dx = 0.
\end{equation}

\begin{Rem}\label{Rem-zero-ave}
The condition \cref{zero-ave} indicates that the periodic perturbations of the conservative quantities, the density and the momentum, should have zero averages. 
Otherwise, if the condition \cref{zero-ave} does not hold, the problem \cref{NS},\cref{ic} turns to be connected with other kinds of Riemann solutions, such like a shock or a combination of multiple waves, which is not the topic of this paper and will be studied in future works.
\end{Rem} 

\begin{Rem}\label{Rem-Omega}
The solution $ (\rho,\uv) $ to the problem \cref{NS}, \cref{ic} is periodic with respect to only $ x_2 $ and $ x_3, $ but not to $ x_1. $ Thus, the problem cannot be studied on the bounded torus $ \Torus^3, $ but on the unbounded domain $ \Omega := \R\times \Torus^2 $ instead. However, the solution keeps oscillating as $ \abs{x_1}\rightarrow +\infty, $ thus a suitable ansatz is needed if we want to use the energy method.
\end{Rem}

Before stating the main theorem, we first  introduce two periodic solutions to \cref{NS} and construct the ansatz of the solution to \cref{NS},\cref{ic}.

Let $ \left( \rho^\pm, \uv^\pm \right)(x,t) = \left( \rho^\pm, u_1^\pm, u_2^\pm, u_3^\pm \right)(x,t) $ denote the unique periodic solutions to \cref{NS} with the periodic initial data
\begin{equation}\label{ic-per}
\left( \rho^\pm, \rho^\pm \uv^\pm \right)(x,0) = \left( \rhob^\pm, \rhob^\pm \ub_1^\pm, 0, 0 \right) + \left(v_0,\wv_0 \right)(x), \quad x\in\Torus^3,
\end{equation}
respectively. 
The global existence and uniqueness of the periodic solutions to \cref{NS}, \cref{ic-per}$ ^\pm $ are standard; see \cref{Lem-per} below. 
Comparing \cref{ic} and \cref{ic-per}$ ^\pm, $ one can see that the periodic solutions $ (\rho^\pm,\uv^\pm) $ share same behaviors as $ (\rho,\uv) $ at far fields, i.e., it holds that
\begin{align*}
(\rho,\rho \uv)(x,0)-(\rho^\pm, \rho^\pm\uv^\pm)(x,0) = (\rhot^r, \rhot^r \uvt^r)(x_1,0) - (\rhob^\pm, \rhob^\pm \uvb^\pm) \quad \forall (x_2, x_3) \in \R^2,
\end{align*} 
which tends to zero as $ x_1\rightarrow \pm \infty, $ respectively. And similarly, if $ \norm{v_0}_{L^\infty(\R^3)} $ is small enough, the difference of the velocity satisfies that
\begin{align*}
\uv(x,0)-\uv^\pm(x,0) & = \frac{(\rhot^r \uvt^r)(x_1,0)+\wv_0(x)}{\rhot^r(x_1,0) +v_0(x)} - \frac{\rhob^\pm \uvb^\pm +\wv_0(x)}{\rhob^\pm +v_0(x)} \\
& = \frac{\left( \rhot^r(x_1,0)-\rhob^\pm \right) \left( v_0(x)\uvb^\pm-\wv_0(x)\right)}{(\rhot^r(x_1,0)+v_0(x))(\rhob^\pm + v_0(x))} + \frac{\rhot^r(x_1,0) (\uvt^r(x_1,0) - \uvb^\pm)}{\rhot^r(x_1,0)+ v_0(x)},
\end{align*}
which yields that
\begin{align*}
\sup_{(x_2,x_3)\in \R^2} \abs{\uv(x,0)-\uv^\pm(x,0)} \leq C \left(\abs{\rhot^r(x_1,0)-\rhob^\pm} + \abs{u_1^r(x_1,0)-\ub_1^\pm}\right) \rightarrow 0 \text{ as } x_1\rightarrow \pm\infty.
\end{align*} 
Observing from this, it is plausible that, when $ t>0, $ the solution $ (\rho, \uv) $ to \cref{NS}, \cref{ic} tends to the periodic solutions $ (\rho^\pm, \uv^\pm) $ as $ x_1\rightarrow \pm\infty, $ respectively as well. 
It is noted that due to the conservative form of \cref{NS} with the condition \cref{zero-ave}, the periodic perturbations
\begin{equation}\label{vw-pm}
	\left(v^\pm, \wv^\pm \right)(x,t) := \left( \rho^\pm, \rho^\pm \uv^\pm \right)(x,t) - \left( \rhob^\pm, \rhob^\pm \ub_1^\pm, 0, 0 \right)
\end{equation}
satisfy that $ \int_{\Torus^3} \left(v^\pm, \wv^\pm\right)(x,t) dx \equiv 0, $ i.e. the periodic solutions $ (\rho^\pm, \rho^\pm\uv^\pm)(x,t) $ still have the same averages $ (\rhob^\pm, \rhob^\pm\uvb^\pm) $ as the initial data \cref{ic-per}$ ^\pm $ for all $ t\geq 0. $ And it is well-known (see \cref{Lem-per}) that the periodic solutions $ (\rho^\pm, \rho^\pm\uv^\pm)(x,t) $ tend to their averages as $ t\to +\infty. $ And this is the reason why the condition \cref{zero-ave} is necessary to ensure the stability of the background rarefaction wave, as stated in \cref{Rem-zero-ave}.
The ansatz is constructed as follows, which is similar to \cite{HY2020}.

\textbf{Ansatz.}
Inspired by the formulas of the background smooth rarefaction wave,
\begin{equation}\label{inspire}
\begin{aligned}
\rhot^r(x_1,t) & = \rhob^- \left(1-\sigma(x_1,t)\right) + \rhob^+ \sigma(x_1,t), \\
\ut_1^r(x_1,t) & = \ub_1^- \left(1-\eta(x_1,t)\right) + \ub_1^+ \eta(x_1,t),
\end{aligned}
\end{equation}
where 
\begin{equation}\label{theta-sig}
\sigma(x_1,t) := \frac{\rhot^r(x_1,t) - \rhob^-}{\rhob^+-\rhob^-}, \quad 
\eta(x_1,t) := \frac{\ut_1^r(x_1,t) - \ub_1^-}{\ub_1^+-\ub_1^-},
\end{equation}
we set the ansatz $ (\rhot,\uvt) = (\rhot, \ut_1, \ut_2,\ut_3) $ as
\begin{equation}\label{ansatz}
\begin{aligned}
\rhot(x,t) & := \rho^-(x,t) \left(1-\sigma(x_1,t)\right) + \rho^+(x,t) \sigma(x_1,t), \\
& = \rhot^r(x_1,t) + v^-(x,t) (1-\sigma(x_1,t)) + v^+(x,t) \sigma(x_1,t), \\
\uvt(x,t) & := \uv^-(x,t) \left(1-\eta(x_1,t)\right) + \uv^+(x,t) \eta(x_1,t), \\
& = \ut_1^r(x_1,t) \E_1 + \zv^-(x,t) (1-\eta(x_1,t)) + \zv^+(x,t) \eta(x_1,t),
\end{aligned}
\end{equation}
where $ \E_1 := (1,0,0) $ is a unit vector and $ \zv^\pm(x,t) := \uv^\pm(x,t) - \uvb^\pm. $ Note that if $ \norm{v_0}_{L^\infty(\Torus^3)} $ is small enough, it follows form \cref{vw-pm} that
\begin{equation}\label{z-pm}
	\zv^\pm(x,t) = \frac{ \wv^\pm(x,t) - v^\pm(x,t) \uvb^\pm }{\rho^\pm(x,t)}.
\end{equation}
Same as the solution to \cref{NS}, \cref{ic}, the ansatz \cref{ansatz} is also periodic with respect to only $ x_2 $ and $ x_3, $ but not to $ x_1. $ Recall the domain
$ \Omega = \R \times \Torus^2. $ We are ready to state the main theorem.
\begin{Thm}\label{Thm}
	Assume in \cref{ic} that the periodic perturbation $ \left(v_0,\wv_0 \right) \in H^5(\Torus^3) $ and satisfies \cref{zero-ave}. Then there exist small $ 
	\delta_0>0 $ and $ \e_0>0 $ such that, if 
	\begin{equation}\label{small}
	\begin{aligned}
	\abs{\rhob^+ -\rhob^-}\leq \delta_0 \quad \text{ and } \quad \norm{\left(v_0,\wv_0 \right)}_{H^5(\Torus^3)} \leq \e_0,
	\end{aligned}
	\end{equation}
	the problem \cref{NS}, \cref{ic} admits a unique global solution $ \left(\rho,\uv\right)(x,t), $ which is periodic with respect to $ x_2 $ and $ x_3, $ and satisfies that
	\begin{equation}\label{exist}
	\begin{aligned}
	(\rho-\rhot, \uv-\uvt) & \in C\left(0,+\infty; H^2(\Omega)\right), \\
	\nabla (\rho-\rhot) & \in L^2\left(0,+\infty; H^1(\Omega)\right), \\
	\nabla (\uv-\uvt) & \in L^2\left(0,+\infty; H^2(\Omega)\right),
	\end{aligned}
	\end{equation}
	with the large time behavior 
	\begin{equation}\label{behavior}
	\sup_{x\in \R^3} \abs{ (\rho,\uv)(x,t)-\left(\rho^r, u_1^r, 0,0 \right)(x_1,t) } \rightarrow 0 \quad \text{as } t\rightarrow +\infty.
	\end{equation}
\end{Thm}


\vspace{0.3cm}

The rest of this paper is organized as follows. In the next section, we will first introduce some useful lemmas and notations. In Section 3, we will prove the a priori estimates and then complete the proof of the main result. In the last appendix, the exponential decay rates of both the periodic solutions and the error terms produced by the ansatz are obtained.


\section{Preliminaries}

\textbf{Notations.}
For convenience, we denote
\begin{equation}
\delta := \abs{\rhob^+ - \rhob^-} \quad \text{ and } \quad \e := \norm{v_0, \wv_0}_{H^5(\Torus^3)},
\end{equation}
and we use the notations
\begin{equation*}
	\norm{\cdot} := \norm{\cdot}_{L^2(\Omega)}, \quad \norm{\cdot}_l := \norm{\cdot}_{H^l(\Omega)} \quad \text{for } l \geq 1.
\end{equation*}

\begin{Lem}[\cite{Matsumura1986}, Lemma 2.1]\label{Lem-rare}
	The smooth rarefaction wave $ \left( \rhot^r, \ut_1^r \right) $ solving \cref{def-sm-rare} satisfies the following properties.
	\begin{itemize}
		\item[i) ] $ \left(\rhot^r, \ut_1^r\right) $ solves the 1-d isentropic Euler equations;
		
		\item[ii) ] $ \p_1  \rhot^r > 0, \p_1  \ut_1^r>0 $ and there exists a constant $ C>0, $ independent of either $ \delta $ or $ t, $ such that $ \abs{\p_1 ^2\ut_1^r} \leq C \p_1  \ut_1^r $ for all $ t\geq 0 $ and $ x_1 \in\R; $
		
		\item[iii) ] $ \sup\limits_{x_1\in \R} \abs{\left(\rhot^r, \ut_1^r\right)-\left(\rho^r, u_1^r\right)}(x_1,t) \rightarrow 0 $ as $ t\rightarrow +\infty; $
		
		\item[iv) ] For any $ p\in [1,+\infty] $ and $ t\geq 0, $ it holds that
		\begin{equation}
		\begin{aligned}
		& \norm{\nabla_{t,x_1} \left(\rhot^r, \ut_1^r\right)}_{L^p(\R)} \leq C \min\left\{ \delta, \delta^{1/p} (1+t)^{-1+1/p} \right\}, \\
		& \norm{\nabla_{t,x_1}^m \left(\rhot^r, \ut_1^r\right)}_{L^p(\R)} \leq C \min\left\{\delta, (1+t)^{-1} \right\} \quad \text{for } m =2,3,
		\end{aligned}
		\end{equation}
		where $ C>0 $ is independent of either $ \delta $ or $ t. $
	\end{itemize}
\end{Lem}

\begin{Lem}\label{Lem-sig-eta}
	The functions $ \sigma $ and $ \eta $ defined in \cref{theta-sig} are smooth and satisfy the following properties.
	\begin{itemize}
		\item[i) ]  $ 0<\sigma(x_1,t), \eta(x_1,t)<1 $ and $ \p_1  \sigma(x_1,t), \p_1  \eta(x_1,t) >0 $ for any $ (x_1,t) \in \R\times [0,+\infty). $
		
		\item[ii) ] For any $ p\in[1,+\infty], $ it holds that
		\begin{align*}
		\norm{\sigma(1-\sigma), \sigma(1-\eta), \eta(1-\sigma), \eta(1-\eta)}_{L^p(\R)} & \leq C (1+t)^{1/p}, \\
		\norm{ \sigma-\eta}_{L^p(\R)} & \leq C \delta (1+t)^{1/p}, \\
		\sup_{t > 0} \norm{\nabla_{t,x_1}^m \left(\sigma,\eta\right)}_{L^p(\R)} & \leq C \quad \text{for } m = 1, 2, 3,
		\end{align*}
		where the constant $ C>0 $ is independent of either $ \delta $ or $ t. $
	\end{itemize}
\end{Lem}
\begin{proof}
	Let $ C $ denote a constant independent of either $ \delta $ or $ t. $
	From \cref{Lem-rare}, it is direct to prove i). And for ii), we prove only $ \sigma(1-\eta) $ and $ \sigma-\eta, $ since the proof of others is similar and the proof of the derivatives is straightforward. 
	
	It follows from \cref{rel-const} that $ C^{-1} \delta \leq \abs{\ub_1^+ -\ub_1^-} \leq C \delta. $ Then by \cref{def-sm-rare}, one has that
	\begin{equation*}
	\begin{aligned}
	& \p_1  \ut_1^r = \frac{\sqrt{p'(\rhot^r)}}{\rhot^r} \p_1  \rhot^r \geq C \p_1  \rhot^r, \\
	& \p_1  \ut_1^r = \left(1+ \frac{p''(\rhot^r) \rhot^r}{2p'(\rhot^r)} \right)^{-1} \p_1  \ot \leq C \p_1  \ot,
	\end{aligned}
	\end{equation*}
	which yields that
	\begin{equation*}
	0< \rhot^r-\rhob^- = \int_{-\infty}^{x_1} \p_y \rhot^r(y,t) dy \leq C \int_{-\infty}^{x_1} \p_y \ot(y,t) dy \leq C \left( \ot- \lambda_2^- \right),
	\end{equation*}
	and similarly,
	\begin{equation*}
	0< \rhob^+- \rhot^r \leq C \left(\lambda_2^+ - \ot  \right).
	\end{equation*}
	Thus, one has that
	\begin{align*}
	\norm{\sigma(1-\eta)}_{L^p(\R)}& \leq C \delta^{-2} \norm{(\rhot^r-\rhob^-)(\ut_1^r-\ub_1^+)}_{L^p(\R)} \\
	& \leq C \delta^{-2} \norm{(\rhot^r-\rhob^-)(\rhot^r-\rhob^+)}_{L^p(\R)} \\
	& \leq C \delta^{-2} \norm{\left(\ot-\lambda_2^- \right) \left(\ot-\lambda_2^+\right)}_{L^p(\R)} \\
	& \leq C (1+t)^{1/p},
	\end{align*}
	where the last inequality can be derived from the characteristic curve method and the fact that $ \abs{ \lambda_2^+ - \lambda_2^- } \leq C\delta. $
	
	Then we show the proof of $ \abs{\sigma-\eta}. $
	Denote $ A(\rho) := \int_{1}^{\rho} \frac{\sqrt{p'(s)}}{s} ds. $	
	It can follow from \cref{rel-const} and \cref{rel-key} that
	\begin{align*}
	\sigma-\eta & = \sigma \left[ 1- \frac{\int_0^1 A'\left( \rhob^- + s (\rhot^r-\rhob^-) \right) ds }{\int_0^1 A'\left( \rhob^- + s (\rhob^+-\rhob^-) \right) ds} \right] \\
	& = \sigma(1-\sigma) b(\rhot^r),
	\end{align*}
	where \begin{equation*}
	b\left(\rhot^r\right) = \left(\rhob^+ - \rhob^-\right) \frac{\int_0^1 \int_0^1 A'' \left( \rhob^- + s(\rhot^r-\rhob^-) + r s (\rhob^+-\rhot^r) \right) dr s ds}{\int_0^1 A'\left( \rhob^- + s (\rhob^+ - \rhob^-) \right) ds},
	\end{equation*}
	which satisfies that $ \norm{b\left(\rhot^r\right)}_{L^\infty(\R\times[0,+\infty))} \leq C \delta. $ 
	Thus, one has that
	\begin{align*}
	\norm{\sigma-\eta}_{L^p(\R)} \leq C \delta \norm{\sigma(1-\sigma)}_{L^p(\R)} \leq C \delta (1+t)^{1/p},
	\end{align*}
	which finishes the proof.

\end{proof}

\vspace{0.3cm}

\begin{Lem}\label{Lem-per}
	If $ \e = \norm{v_0,\wv_0}_{H^5(\Torus^3)} $ is small enough, then the global periodic solution $ \left(\rho^\pm, \uv^\pm\right) $ to the problem \cref{NS}, \cref{ic-per}$ ^\pm $ exists and satisfies $$ \int_{\Torus^3} \left(\rho^\pm - \rhob^\pm\right)(x,t) dx = 0, \quad \int_{\Torus^3} \left(\rho^\pm \uv^\pm - \rhob^\pm \uvb^\pm\right)(x,t) dx = 0 \qquad \forall t\geq 0. $$ 
	Moreover, the periodic perturbations \cref{vw-pm} and \cref{z-pm} satisfy that
	\begin{equation}\label{decay-periodic}
	\norm{ \left(v^\pm, \zv^\pm \right)}_{W^{3,\infty}(\Torus^3)}(t) = 
	\norm{ \left(\rho^\pm, \uv^\pm \right) - \left(\rhob^\pm, \ub_1^\pm \E_1 \right) }_{W^{3,\infty}(\Torus^3)}(t) \leq C \e e^{-2\alpha t}, \quad t\geq 0,
	\end{equation}
	where the constants $ C>0 $ and $ \alpha>0 $ are independent of either $ t $ or $ \e. $
\end{Lem}
By using the energy method with the aid of the Poincar\'{e} inequality on $ \Torus^3, $ the proof of \cref{Lem-per} is standard and is placed in the appendix.

Although the ansatz \cref{ansatz} is not a solution to \cref{NS}, its error terms
\begin{equation}\label{source}
\begin{aligned}
h_0 & := \p_t \rhot + \dv(\rhot \uvt), \\
\h = \left(h_1,h_2,h_3\right)
& := \p_t \left(\rhot \uvt\right) + \dv \left(\rhot \uvt \otimes \uvt\right) + \nabla p(\rhot) - \mu \lap \uvt - (\mu+\lambda) \nabla \dv \uvt \\
& = h_0 \uvt + \rhot\p_t\uvt + \rhot\uvt\dnab\uvt + \nabla p(\rhot) - \mu \lap \uvt - (\mu+\lambda) \nabla \dv \uvt
\end{aligned}
\end{equation}
decay exponentially fast with respect to time. More precisely, it holds that

\begin{Lem}\label{Lem-h}
	Under the assumptions of \cref{Thm}, the error terms \cref{source} satisfy that
	\begin{equation}\label{est-h}
	\begin{aligned}
	\norm{h_0}_{W^{2,p}(\Omega)} + \norm{\h+\left(2\mu+\lambda\right)\p_1 ^2 \ut_1^r \E_1}_{W^{1,p}(\Omega)} & \leq C\e e^{-\alpha t}, \quad p\in[1,+\infty],
	\end{aligned}
	\end{equation}
	where $ \alpha>0 $ is the constant in \cref{Lem-per} and the constant $ C>0 $ is independent of either $ \delta, \e $ or $ t. $
\end{Lem}
The proof of \cref{Lem-h} is similar to \cite[Lemma 2.3]{HY2020}, which is based on Lemmas \ref{Lem-sig-eta} and \ref{Lem-per}. We still place it in the appendix for brevity.

\vspace{0.3cm}

As indicated in \cite{HY2020}, the functions which are integrable on the domain $ \Omega = \R\times\Torus^2 $ without any additional boundary conditions do not satisfy the 3-d Gagliardo-Nirenberg (G-N) inequalities in general (any 1-d function $ f(x_1) \in C_c^\infty(\R), $ which is periodic with respect to $ x_2 $ and $ x_3, $ is a counterexample). To solve the problem \cref{NS}, \cref{ic} on the domain $ \Omega = \R\times\Torus^2, $ we need the following G-N type inequality.

\begin{Lem}[\cite{HY2020}, Theorem 1.4]\label{Lem-GN}
	Assume that $ u(x) $ is in the $ L^q(\Omega) $ space with $ \nabla^m u \in L^r(\Omega), $ where $ 1\leq q,r\leq +\infty $ and $ m\geq 1, $ and $ u $ is periodic with respect to $ x_2 $ and $ x_3. $ Then there exists a decomposition $ u(x) = \sum\limits_{k=1}^{3} u^{(k)}(x) $ such that each $ u^{(k)} $ satisfies the $ k $-dimensional G-N inequality,
	\begin{equation}\label{G-N-type-1}
		\norm{\nabla^j u^{(k)}}_{L^p(\Omega)} \leq C \norm{\nabla^{m} u}_{L^{r}(\Omega)}^{\theta_k} \norm{u}_{L^{q}(\Omega)}^{1-\theta_k},
	\end{equation}
	where $ 0\leq j< m $ is any integer and $ 1\leq p \leq +\infty $ is any number, satisfying 
	$$ \frac{1}{p} = \frac{j}{k} + \left(\frac{1}{r}-\frac{m}{k}\right) \theta_k + \frac{1}{q}\left(1-\theta_k\right) \quad \text{with} \quad \frac{j}{m} \leq \theta_k \leq 1. $$
	Moreover, it holds that 
	\begin{equation}\label{G-N-type-2}
	\norm{\nabla^j u}_{L^p(\Omega)} \leq C \sum_{k=1}^3 \norm{\nabla^{m} u}_{L^{r}(\Omega)}^{\theta_k} \norm{u}_{L^{q}(\Omega)}^{1-\theta_k}.
	\end{equation}
	The constants $ C>0 $ in \cref{G-N-type-1,G-N-type-2} are independent of $ u. $
\end{Lem}

\vspace{0.3cm}

\section{Reformulation of the problem and proof}

We first show the equations satisfied by the perturbation terms, 
\begin{equation}\label{def-pertur}
	\phi := \rho-\rhot \quad \text{ and } \quad \psi= \left( \psi_1,\psi_2,\psi_3\right) := \uv-\uvt.
\end{equation}
It follows from \cref{NS} and \cref{source} that
\begin{align}
& \p_t \phi + \rho \dv \psi + \uv \cdot \nabla \phi + \phi \dv \uvt + \nabla \rhot \cdot \psi = - h_0, \label{eq-pertur-1} \\
& \rho \p_t \psi + \rho \uv \cdot \nabla \psi + \rho \psi \cdot \nabla \uvt + p'(\rho) \nabla \phi + \Big(p'(\rho) - \frac{\rho}{\rhot} p'(\rhot)\Big) \nabla \rhot \notag \\ 
& \qquad\qquad\qquad - \mu \lap \psi - (\mu+\lambda) \nabla \dv \psi = \f - \phi \g + \left(2\mu+\lambda\right)\p_1 ^2 \ut_1^r \E_1, \label{eq-pertur-2}
\end{align}
where 
\begin{align*}
\f = (f_1,f_2,f_3)^T & = h_0 \uvt - \h-\left(2\mu+\lambda\right)\p_1 ^2 \ut_1^r \E_1, \\
\g = (g_1,g_2,g_3)^T & = \frac{1}{\rhot} \left[ \mu\lap \uvt + (\mu+\lambda) \nabla \dv \uvt + \h - h_0 \uvt \right], \\
& = \frac{1}{\rhot} \big[ \mu\lap (\uvt-\ut_1^r \E_1) + (\mu+\lambda) \nabla \dv (\uvt-\ut_1^r \E_1) - \f \big],
\end{align*}
which satisfy from \cref{Lem-h} that
\begin{equation}\label{g-f}
\begin{aligned}
\norm{\f}_{W^{1,p}(\Omega)} & \leq C\e e^{-\alpha t} \quad \forall p\in[1,+\infty], \\
\norm{\g}_{W^{1,\infty}(\Omega)} & \leq C\e e^{-\alpha t}.
\end{aligned}
\end{equation}
From \cref{ic,z-pm,ansatz}, one has that
\begin{align*}
\phi(x,0) & = \rho(x,0)- \rhot^r(x_1,0) - v_0(x)(1-\sigma(x_1,0)) - v_0(x) \sigma(x_1,0) = 0, \\
\psi(x,0) & = \frac{\rhot^r(x_1,0) \ut_1^r(x_1,0) \E_1 + \wv_0(x)}{\rhot^r(x_1,0)+v_0(x)} -  \ut_1^r(x_1,0) \E_1 \\
& \quad - \frac{\wv_0(x)- \ub_1^-v_0(x) \E_1}{\rhob^- +v_0(x)} (1-\eta(x_1,0)) - \frac{\wv_0(x)-\ub_1^+ v_0(x)\E_1}{\rhob^+ +v_0(x)}\eta(x_1,0) \\
& = \left( \frac{1}{\rhot^r(x_1,0)+v_0(x)} - \frac{1-\eta(x_1,0)}{\rhob^-+v_0(x)} - \frac{\eta(x_1,0)}{\rhob^+ +v_0(x)} \right) \wv_0(x) \\
&\quad - \left( \frac{\ut_1^r(x_1,0)}{\rhot^r(x_1,0)+v_0(x)} - \ub_1^- \frac{1-\eta(x_1,0)}{\rhob^-+v_0(x)} - \ub_1^+ \frac{\eta(x_1,0)}{\rhob^+ +v_0(x)} \right) v_0(x) \E_1 \\
& = \frac{\rhob^+-\rhob^-}{\rhot^r(x_1,0) + v_0(x)} \left[ \frac{(\eta(1-\sigma))(x_1,0)}{\rhob^+ + v_0(x)} - \frac{(\sigma(1-\eta))(x_1,0)}{\rhob^- + v_0(x)} \right] \wv_0(x) \\
& \quad - \frac{\rhob^+-\rhob^-}{\rhot^r(x_1,0) + v_0(x)} \left[ \frac{\ub_1^- (\eta(1-\sigma))(x_1,0)}{\rhob^+ + v_0(x)} - \frac{\ub_1^+ (\sigma(1-\eta))(x_1,0)}{\rhob^- + v_0(x)} \right] v_0(x) \E_1 \\
& := q_1(x) \wv_0(x) - q_2(x) v_0(x) \E_1.
\end{align*} 
If $ \delta_0\leq 1 $ and $ \e_0 $ is small enough, it follows from \cref{Lem-sig-eta} that
\begin{equation}\label{psi0}
\norm{\psi(\cdot,0)}_2 \leq  \norm{\wv_0, v_0}_{W^{2,\infty}} \norm{q_1,q_2}_2 \leq C_0\delta \e \leq C_0\e,
\end{equation}
where $ C_0>0 $ is independent of $ \e $ or $ \delta. $

\vspace{0.3cm}

Now we aim to solve the problem \cref{eq-pertur-1,eq-pertur-2} with the initial data
\begin{equation}\label{ic-pertur}
\begin{aligned}
\phi(x,0) & = \phi_0(x) = 0, \\
\psi(x,0) & = \psi_0(x) = q_1(x) \wv_0(x) - q_2(x) v_0(x) \E_1,
\end{aligned}
\end{equation}
where $ \psi_0 $ satisfies \cref{psi0}.
The proof consists of the a priori estimates (\cref{Thm-apriori}) and the local existence (\cref{Prop-local}). For $ T>0, $ denote
\begin{equation}\label{apriori}
	N(T) := \Big\{ \sup_{t\in(0,T)} \norm{\phi,\psi}_2^2 + \int_0^T \left(\norm{\nabla \phi}_1^2 + \norm{\nabla\psi}_2^2\right) dt \Big\}^{\frac{1}{2}}.
\end{equation}


\begin{Thm}\label{Thm-apriori}
Under the assumptions of \cref{Thm}, let $ T>0 $ and
$$ (\phi,\psi) \in C\left(0,T; H^2(\Omega)\right) \quad \text{ with } \quad \nabla\phi \in L^2\left(0,T; H^1(\Omega)\right), \quad \nabla \psi \in L^2\left(0,T; H^2(\Omega)\right),  $$
solve the problem \cref{eq-pertur-1}, \cref{eq-pertur-2} 
with initial data $ (\phi_0,\psi_0)(x) \in H^2(\Omega). $  
Then there exist $ \e_0>0, \delta_0>0 $ and $ \nu_0>0 $ such that if $ \e <\e_0, \delta<\delta_0 $ and $ N(T)<\nu_0, $ then 
\begin{equation}\label{est-apriori}
N(T)^2 \leq C\norm{ \phi_0,\psi_0 }_2^2 + C(\e +\delta^{\frac{1}{4}}),
\end{equation}
where the constant $ C>0 $ is independent of $ \e, \delta $ or $ N(T). $
\end{Thm}
For convenience, in the remaining part of this paper, we let $ C>0 $ denote a generic constant which is independent of $ \e, \delta $ or $ N(T). $
Under the assumptions of \cref{Thm-apriori}, it follows from \cref{Lem-GN} that
\begin{align*} 
& \sup_{t\in(0,T)} \norm{\phi,\psi}_{L^\infty(\Omega)} \\
&\quad \leq C\sup_{t\in(0,T)} \left\{ \norm{\nabla(\phi,\psi)}^{\frac{1}{2}} \norm{\phi,\psi}^{\frac{1}{2}} + \norm{\nabla (\phi,\psi)} + \norm{\nabla^2 (\phi,\psi)}^{\frac{3}{4}} \norm{\phi,\psi}^{\frac{1}{4}} \right\} \\
&\quad \leq C \sup_{t\in(0,T)} \norm{\phi,\psi}_2 \leq C N(T).
\end{align*}
Thus, one can first choose $ 0<\nu_0<1 $ small enough such that
\begin{equation}\label{bdd-rho-u}
\frac{1}{2}\rhob^- \leq \inf_{\substack{x\in\Omega \\ t\in(0,T)}} \rho(x,t)  \leq \sup_{\substack{x\in\Omega \\ t\in(0,T)}} \rho(x,t) \leq 2 \rhob^+ \quad \text{ and } \quad \sup_{\substack{x\in\Omega \\ t\in(0,T)}} \abs{\uv(x,t)} \leq C.
\end{equation}

\begin{Lem}\label{Lem-est-0}
Under the assumptions of \cref{Thm-apriori}, if $ \e_0>0 $ and $ \nu_0>0 $ are small enough, then
\begin{equation}\label{eq-est-0}
\sup_{t\in(0,T)} \norm{\phi, \psi}^2 + \int_0^T \norm{ \left(\p_1 \ut_1^r\right)^{\frac{1}{2}} \left(\phi, \psi_1\right) }^2 dt + \int_0^T \norm{\nabla \psi}^2  dt \leq C\norm{\phi_0, \psi_0}^2 + C (\e+\delta^{\frac{1}{4}}).
\end{equation}
\end{Lem}

\begin{proof}

Define 
\begin{equation*}
\Phi(\rho,\rhot)= \int_{\rhot}^\rho \frac{p(s)-p(\rhot)}{s^2} ds = \frac{\gamma}{(\gamma-1)\rho} \left[p(\rho)-p(\rhot) -p'(\rhot) (\rho-\rhot)\right],
\end{equation*}
which satisfies $ C^{-1} \abs{\phi}^2 \leq \Phi(\rho,\rhot) \leq C \abs{\phi}^2 $ for some constant $ C>0, $ and
\begin{equation*}
\p_\rho \Phi = \frac{p(\rho)-p(\rhot)}{\rho^2}, \qquad \p_{\rhot} \Phi = p'(\rhot) \Big(\frac{1}{\rho} - \frac{1}{\rhot}\Big).
\end{equation*}
Then multiplying $ \Phi $ on  the first equation of \cref{NS} yields that
\begin{equation*}
\begin{aligned}
\p_t(\rho\Phi) + \dv(\rho \Phi \uv) & = \rho \left(\p_t\Phi + \uv\dnab\Phi\right) \\
& = \rho \p_\rho \Phi (\p_t \rho+\uv\dnab\rho) + \rho \p_{\rhot} \Phi (\p_t \rhot +  \uv \dnab\rhot) \\
& = - \left[p(\rho)-p(\rhot)\right] \dv \uv - p'(\rhot) \phi \Big[ \frac{1}{\rhot} \left( h_0+\psi\dnab\rhot \right) -\dv \uvt \Big] \\
& = -\left[p(\rho)-p(\rhot) -p'(\rhot) \phi \right] \dv \uvt -\dv\left[ \left(p(\rho)-p(\rhot)\right) \psi \right] \\
& \quad + \psi\dnab \left(p(\rho)-p(\rhot)\right) -\frac{p'(\rhot)}{\rhot}\phi \left( h_0+\psi\dnab\rhot \right),
\end{aligned}
\end{equation*}
which gives that
\begin{equation}\label{eq-1}
\begin{aligned}
& \p_t(\rho\Phi) + \left[p(\rho)-p(\rhot) -p'(\rhot) \phi \right] \p_1 \ut_1^r - p'(\rho) \psi\dnab \phi \\
& \quad = - \dv Q_1 + \Big(p'(\rho)-\frac{\rho}{\rhot}p'(\rhot)\Big) \psi\dnab\rhot \\
& \qquad -\left[p(\rho)-p(\rhot) -p'(\rhot) \phi \right]  \left(\dv\uvt-\p_1 \ut_1^r\right) -\frac{p'(\rhot)}{\rhot} h_0 \phi,
\end{aligned}
\end{equation}
where $ Q_1 = \rho \Phi \uv + \left(p(\rho)-p(\rhot)\right) \psi, $ and it follows from \cref{Lem-per} that
\begin{equation*}
\begin{aligned}
& \norm{\dv \uvt-\p_1  \ut_1^r}_{L^\infty(\Omega)} \leq C \norm{\zv^\pm}_{W^{1,\infty}(\Omega)} \leq C\e e^{-\alpha t}.
\end{aligned}
\end{equation*}
Multiplying $ \cdot\psi $ on both sides of \cref{eq-pertur-2} yields that 
\begin{equation*}
\begin{aligned}
	& \p_t \Big( \frac{1}{2} \rho \abs{\psi}^2 \Big) - \frac{1}{2} \p_t \rho \abs{\psi}^2 + \rho (\uv \cdot \nabla \psi) \cdot \psi + \rho (\psi\cdot\nabla \uvt) \cdot\psi  + p'(\rho) \nabla \phi \cdot \psi \\
	& \quad  + \Big(p'(\rho)-\frac{\rho}{\rhot}p'(\rhot)\Big) \nabla\rhot \cdot \psi + \mu \abs{\nabla\psi}^2 - \frac{\mu}{2} \dv\left(\nabla\abs{\psi}^2\right) + (\mu+\lambda)\abs{\dv \psi}^2 \\
	& \quad - (\mu+\lambda) \dv\left(\psi\dv\psi\right) =  \f\cdot \psi - \phi \g\cdot\psi + \left(2\mu+\lambda\right)\p_1 ^2 \ut_1^r \psi_1.
\end{aligned}
\end{equation*}
Note that
\begin{align*}
-\frac{1}{2} \p_t\rho \abs{\psi}^2 + \rho (\uv\cdot\nabla)\psi \cdot\psi & = -\frac{1}{2} \left[ \p_t\rho + \dv(\rho \uv) \right] \abs{\psi}^2 + \dv\Big( \frac{1}{2} \rho \abs{\psi}^2 \uv \Big) \\
& = \dv\Big( \frac{1}{2} \rho \abs{\psi}^2 \uv \Big), 
\end{align*}
and
\begin{align*}
\rho (\psi\cdot\nabla\uvt ) \cdot\psi & = \rho (\psi\dnab (\ut_1^r \E_1)) \cdot \psi + \rho (\psi\dnab (\uvt- \ut_1^r \E_1)) \cdot \psi  \\
& = \rho \p_1  \ut_1^r \psi_1^2 + \underbrace{\rho (\psi\dnab (\uvt- \ut_1^r \E_1)) \cdot \psi}_{I_1}.
\end{align*}
By \cref{ansatz} and \cref{Lem-per}, one has that
\begin{equation*}
\abs{I_1} \leq C \norm{\zv^\pm}_{W^{1,\infty}(\Omega)} \abs{\psi}^2 \leq C\e e^{-\alpha t} \abs{\psi}^2.
\end{equation*}
Then one has
\begin{equation}\label{eq-2}
\begin{aligned}
& \p_t \Big( \frac{1}{2} \rho \abs{\psi}^2 \Big) + \mu \abs{\nabla\psi}^2 + (\mu+\lambda)\abs{\dv \psi}^2 + p'(\rho)\psi\dnab \phi + \rho \p_1 \ut_1^r \psi_1^2 \\
&\quad = \f\cdot\psi - \phi\g\cdot\psi + \left(2\mu+\lambda\right) \p_1 ^2 \ut_1^r \psi_1 +  \dv Q_2 - I_1 - \Big(p'(\rho)-\frac{\rho}{\rhot}p'(\rhot)\Big)\nabla\rhot\cdot\psi,
\end{aligned}
\end{equation}
where $ Q_2 = \frac{\mu}{2} \nabla \abs{\psi}^2 + (\mu+\lambda) \psi\dv\psi - \frac{1}{2}\rho \abs{\psi}^2 \uv. $ 
Summing \cref{eq-1,eq-2} together, integrating the resulting equation over $ \Omega\times(0,T) $ and combining \cref{g-f}, one has that
\begin{align} 
& \sup_{t\in(0,T)} \norm{\phi, \psi}^2 + \int_0^T \norm{\nabla \psi}^2 dt + \int_0^T \norm{\left(\p_1 \ut_1^r\right)^{\frac{1}{2}}\left(\phi,\psi_1\right)}^2 dt \notag \\
&\quad \leq C\norm{\phi_0, \psi_0}^2 + C \int_0^T \Big\{ \e e^{-\alpha t} \left(\norm{\phi}^2 + \norm{\psi}^2 \right) + \norm{h_0}\norm{\phi} \notag \\
& \qquad\qquad\qquad\quad + \left(\norm{\f}+ \norm{\g}_{L^\infty(\Omega)} \norm{\phi}\right)\norm{\psi} + \int_\Omega \abs{\p_1 ^2 \ut^r_1} \abs{\psi_1} dx \Big\} dt \notag \\
&\quad \leq C\norm{\phi_0, \psi_0}^2 + C\e \sup_{t\in(0,T)} \norm{\phi, \psi}^2 + C\e + C \int_0^T \int_\Omega \abs{\p_1 ^2 \ut^r_1} \abs{\psi} dx dt. \label{eq-3}
\end{align}
Decompose $ \psi=\sum\limits_{k=1}^3 \psi^{(k)} $ as in \cref{Lem-GN} such that $ \psi^{(k)} $ satisfies the $ k $-dimensional G-N inequalities. Then it follows from \cref{Lem-rare} that the last term in \cref{eq-3} satisfies that
\begin{align}
& C \int_0^T \int_\Omega \abs{\p_1 ^2 \ut^r_1} \abs{\psi} dx dt \leq C \int_0^T \int_\Omega \sum\limits_{k=1}^3 \abs{\p_1 ^2 \ut^r_1} \abs{\psi^{(k)}} dx dt \notag \\
&\quad \leq C \int_0^T \Big[ \norm{\p_1 ^2 \ut^r_1}_{L^1(\Omega)} \left(\norm{\psi^{(1)}}_{L^\infty(\Omega)} + \norm{\psi^{(2)}}_{L^\infty(\Omega)} \right) + \norm{\p_1 ^2 \ut^r_1}_{L^{\frac{6}{5}}(\Omega)} \norm{\psi^{(3)}}_{L^6(\Omega)} \Big] dt \notag \\
&\quad \leq C \int_0^T \min\{\delta, (1+t)^{-1}\} \Big( \norm{\nabla\psi^{(1)}}^{\frac{1}{2}} \norm{\psi^{(1)}}^{\frac{1}{2}} + \norm{\nabla\psi^{(2)}} + \norm{\nabla\psi^{(3)}} \Big) dt \notag \\
&\quad \leq  C \int_0^T \min\{\delta,(1+t)^{-1}\}^2 dt + \frac{1}{2} \int_0^T \norm{\nabla \psi}^2 dt  + C \int_0^T \min\{\delta,(1+t)^{-1}\}^{\frac{4}{3}} \norm{\psi}^{\frac{2}{3}} dt \notag \\
&\quad \leq  C\delta + \frac{1}{2} \int_0^T \norm{\nabla\psi}^2 dt + \frac{1}{4} \int_0^T (1+t)^{-\frac{3}{2}} \norm{\psi}^2 dt + C \int_0^T (1+t)^{\frac{3}{4}} \min\{\delta,(1+t)^{-1}\}^2 dt \notag \\
&\quad \leq \frac{1}{2} \int_0^T \norm{\nabla\psi}^2 dt + \frac{1}{2} \sup_{t\in(0,T)} \norm{\psi}^2 + C\delta^{\frac{1}{4}}. \label{eq-4}
\end{align}
Collecting \cref{eq-3,eq-4}, one can obtain \cref{eq-est-0} if $ \e>0 $ is small enough.
	
\end{proof}

\vspace{0.3cm}

\begin{Lem}\label{Lem-est-1}
Under the assumptions of \cref{Thm-apriori}, if $ \delta_0>0, \e_0>0 $ and $ \nu_0>0 $ are small enough, then one has that
\begin{equation}\label{eq-est-1}
	\sup_{t\in(0,T)} \norm{\nabla\phi}^2 + \int_0^T \norm{\nabla\phi}^2 dt \leq C \norm{\phi_0}_1^2 + C \norm{\psi_0}^2 + C (\e+\delta^{\frac{1}{4}}) + C \nu_0 \int_0^T \norm{\nabla^3 \psi}^2 dt.
\end{equation}
\end{Lem}

\begin{proof}
	
Taking the gradient $ \nabla $ on \cref{eq-pertur-1} and then multiplying the result by $ \cdot \frac{\nabla\phi}{\rho^2}, $ one has that
\begin{align}
& \p_t \Big( \frac{\abs{\nabla \phi}^2}{2\rho^2}\Big)+ \frac{\abs{\nabla \phi}^2}{\rho^3} \p_t\rho + \frac{\dv \psi}{\rho^2} \left(\nabla \rhot + \nabla\phi\right) \cdot \nabla\phi + \frac{1}{\rho} \nabla \dv \psi \cdot \nabla \phi  \notag \\
& \quad + \frac{1}{2\rho^2} (\uv\cdot \nabla)  \abs{\nabla \phi}^2 + \frac{1}{\rho^2} \nabla\phi \cdot \nabla \uv\nabla\phi  + \frac{\dv\uvt}{\rho^2} \abs{\nabla\phi}^2 + \frac{\phi}{\rho^2} \nabla\dv\uvt \dnab \phi \label{eq-4-1} \\
& \quad + \frac{1}{\rho^2} \nabla\phi\dnab^2\rhot \psi + \frac{1}{\rho^2} \nabla\phi\dnab\psi \nabla\rhot = - \frac{1}{\rho^2} \nabla h_0 \cdot\nabla\phi. \notag
\end{align}
Note that in \cref{eq-4-1}, the sum of the second term on the first line and the first term on the second line satisfies that
\begin{align*}
\frac{\abs{\nabla \phi}^2}{\rho^3} \p_t\rho+\frac{1}{2\rho^2} (\uv\cdot \nabla)  \abs{\nabla \phi}^2 & = \frac{\abs{\nabla \phi}^2}{\rho^3} \p_t\rho + \dv \Big(\frac{\abs{\nabla\phi}^2}{2\rho^2}  \uv\Big) - \dv \Big(\frac{\uv}{2\rho^2}\Big) \abs{\nabla\phi}^2 \\
& = \dv \Big( \frac{\abs{\nabla\phi}^2}{2\rho^2} \uv\Big) - \frac{3}{2\rho^2} \dv \uvt \abs{\nabla\phi}^2 - \frac{3}{2\rho^2} \dv \psi \abs{\nabla\phi}^2.
\end{align*}
Then \cref{eq-4-1} yields that
\begin{equation}\label{eq-5}
\begin{aligned}
& \p_t \Big(\frac{\abs{\nabla \phi}^2}{2\rho^2}\Big) -\frac{\dv \uvt}{2\rho^2} \abs{\nabla\phi}^2 - \frac{\dv \psi}{2\rho^2} \abs{\nabla\phi}^2 + \frac{\dv \psi}{\rho^2} \nabla \rhot \cdot \nabla\phi + \frac{1}{\rho} \nabla \dv \psi \cdot \nabla \phi \\
& \qquad + \frac{1}{\rho^2} \Big[ \nabla\phi \cdot \nabla \left(\uvt+\psi\right) \nabla\phi + \phi\nabla\dv\uvt \dnab \phi + \nabla\phi\dnab^2\rhot \psi + \nabla\phi\dnab\psi \nabla\rhot \Big] \\
&\quad = -\dv \Big( \frac{\abs{\nabla\phi}^2}{2\rho^2} \uv\Big) - \frac{1}{\rho^2} \nabla h_0 \cdot\nabla\phi.
\end{aligned}
\end{equation}
Multiplying $ \cdot\frac{\nabla\phi}{\rho} $ on \cref{eq-pertur-2} yields that
\begin{align*}
& \p_t \psi \cdot \nabla \phi + (\uv\cdot \nabla \psi) \cdot \nabla \phi + (\psi\cdot\nabla \uvt)\cdot\nabla\phi + \frac{p'(\rho)}{\rho} \abs{\nabla \phi}^2  \\
& \qquad  + \Big(\frac{p'(\rho)}{\rho}-\frac{p'(\rhot)}{\rhot}\Big) \nabla \rhot \cdot \nabla\phi - \frac{1}{\rho} \left[ \mu \lap \psi + (\mu+\lambda) \nabla\dv \psi \right] \cdot \nabla \phi \\
& \quad = \frac{1}{\rho} \f\cdot\nabla\phi - \frac{\phi}{\rho} \g\cdot\nabla \phi + \frac{2\mu+\lambda}{\rho} \p_1 ^2\ut_1^r \p_1 \phi. 
\end{align*}
Note that
\begin{align}
\p_t \psi \cdot \nabla \phi & = \p_t \left(\psi\cdot\nabla \phi \right) - \psi\cdot \nabla \p_t \phi \notag \\
& = \p_t \left(\psi\cdot\nabla \phi \right) - \dv (\psi\p_t\phi) \notag \\
&\quad - \dv\psi\big( h_0 + \rho \dv \psi + \uv\cdot \nabla\phi + \phi 
\dv\uvt + \nabla\rhot \cdot\psi \big), \label{eq-6}
\end{align}
and
\begin{align}
& \frac{1}{\rho} \left[ \mu \lap \psi + (\mu+\lambda) \nabla\dv \psi \right] \cdot \nabla \phi \notag \\
& \quad = \frac{2\mu+\lambda}{\rho} \nabla\dv \psi \cdot \nabla\phi + \frac{\mu}{\rho} \left( \lap \psi - \nabla\dv \psi \right) \cdot \nabla\phi \notag \\
& \quad = \frac{2\mu+\lambda}{\rho} \nabla\dv \psi \cdot \nabla\phi + \frac{\mu}{\rho} \dv \left(\nabla \phi \times \text{curl} \psi \right) \notag\\
& \quad = \frac{2\mu+\lambda}{\rho} \nabla\dv \psi \cdot \nabla\phi + \mu \dv \Big( \frac{ \nabla \phi \times \text{curl} \psi }{\rho}  \Big) + \frac{\mu \nabla \rho \cdot \left(\nabla \phi \times \text{curl} \psi \right) }{\rho^2} \label{eq-7} \\
& \quad = \frac{2\mu+\lambda}{\rho} \nabla\dv \psi \cdot \nabla\phi + \mu \dv \Big(\frac{ \nabla \phi \times \text{curl} \psi }{\rho}  \Big) + \frac{\mu \nabla \rhot \cdot \left(\nabla \phi \times \text{curl} \psi \right) }{\rho^2}, \notag
\end{align}
here and hereafter $ ``\times" $ denotes the external product of vectors. Thus it holds that
\begin{align}
& \p_t (\psi\dnab\phi) + \frac{p'(\rho)}{\rho} \abs{\nabla \phi}^2 - \frac{2\mu+\lambda}{\rho} \nabla\dv \psi \cdot \nabla\phi \notag\\
& \quad = \dv \Big( \psi\p_t \phi + \frac{\mu}{\rho} \nabla\phi 
\times \text{curl} \psi \Big) + \dv\psi\big( h_0 + \rho \dv \psi + \uv\cdot \nabla\phi + \phi 
\dv\uvt + \nabla\rhot \cdot\psi \big) \notag\\
& \qquad -(\uv\cdot \nabla \psi) \cdot \nabla \phi - (\psi\cdot\nabla \uvt)\cdot\nabla\phi - \Big(\frac{p'(\rho)}{\rho}-\frac{p'(\rhot)}{\rhot}\Big) \nabla \rhot \cdot \nabla\phi+ \frac{\mu \nabla \rhot \cdot \left(\nabla \phi \times \text{curl} \psi \right) }{\rho^2} \notag\\
& \qquad + \frac{1}{\rho} \f\cdot\nabla\phi - \frac{\phi}{\rho} \g\cdot\nabla \phi + \frac{2\mu+\lambda}{\rho} \p_1 ^2\ut_1^r \p_1 \phi. \label{eq-9}
\end{align}
Then by multiplying the constant $ 2\mu+\lambda $ on \cref{eq-5} and then adding the result onto \cref{eq-9}, one can get that
\begin{equation}\label{eq-8}
\p_t \Big( \frac{2\mu+\lambda}{2\rho^2}\abs{\nabla \phi}^2 + \psi\cdot\nabla \phi \Big) + \frac{p'(\rho)}{\rho} \abs{\nabla\phi}^2 = \dv Q_3 + \sum_{j=2}^5 I_j.
\end{equation}
where $ Q_3 = - \frac{\abs{\nabla\phi}^2}{2\rho^2} \uv + \psi\p_t\phi + \frac{\mu}{\rho} \nabla\phi 
\times \text{curl} \psi $ and
\begin{align*}
I_2 =~& \frac{2\mu+\lambda}{\rho^2} \Big( \frac{1}{2} \dv\uvt \abs{\nabla\phi}^2 - \dv \psi \nabla \rhot \cdot \nabla\phi - \nabla\phi \cdot \nabla \uvt \nabla\phi  - \nabla\phi\cdot\nabla\psi\nabla\rhot \Big) \\
&  + \rho (\dv \psi)^2 + \dv\psi \uv\cdot \nabla\phi  - (\uv \cdot \nabla \psi) \cdot \nabla \phi -\frac{\mu}{\rho^2} \nabla\rhot\cdot \left(\nabla\phi \times\text{curl} \psi \right), \\
I_3 =~& - \frac{2\mu+\lambda}{\rho^2} \left( \phi\nabla\dv\uvt\dnab\phi + \nabla\phi\cdot \nabla^2\rhot\psi \right) + \dv\psi \left( \phi\dv\uvt + \nabla\rhot \cdot \psi\right) \\
&  - (\psi\dnab\uvt)\dnab\phi - \Big(\frac{p'(\rho)}{\rho}-\frac{p'(\rhot)}{\rhot}\Big) \nabla \rhot \cdot \nabla\phi - \frac{\phi}{\rho} \g\dnab\phi, \\
I_4 =~&\frac{2\mu+\lambda}{\rho^2} \Big(\frac{1}{2} \dv \psi \abs{\nabla\phi}^2 -  \nabla\phi \cdot \nabla \psi\nabla\phi\Big), \\
I_5 =~& - \frac{2\mu+\lambda}{\rho^2} \nabla h_0 \dnab\phi + \dv\psi h_0  + \frac{1}{\rho} \f \dnab\phi + \frac{2\mu+\lambda}{\rho} \p_1 ^2\ut_1^r \p_1 \phi.
\end{align*}
For $ I_2, $ first note that
\begin{equation*}
\begin{aligned}
\norm{\nabla\uvt}_{L^\infty(\Omega)} & \leq \norm{\p_1 \ut^r_1}_{L^\infty(\R)} + C \norm{\zv^\pm}_{W^{1,+\infty}(\Omega)} \leq C(\delta+\e),\\
\norm{\nabla\rhot} _{L^\infty(\Omega)}
& \leq \norm{\p_1  \rhot^r} _{L^\infty(\R)} + C \norm{v^\pm}_{W^{1,+\infty}(\Omega)} \leq C(\delta+\e).
\end{aligned}
\end{equation*}
Then it follows from \cref{bdd-rho-u} that
\begin{equation}\label{I-2}
\int_0^T \norm{I_2}_{L^1(\Omega)} dt \leq \Big(\frac{1}{8}+C(\delta+\e)\Big) \int_0^T \int_\Omega \frac{p'(\rho)}{\rho} \abs{\nabla\phi}^2 dxdt + C \int_0^T \norm{\nabla\psi}^2 dt.
\end{equation}
\vspace{0.2cm}

For $ I_3, $ first note that $ \abs{\nabla\dv\uvt} \leq \abs{\p_1^2 \ut_1^r} + C\e e^{-\alpha t}, $ $ \abs{\nabla^2\rhot } \leq \abs{\p_1^2 \rhot^r} + C \e e^{-\alpha t} $ and
\begin{equation*}
\begin{aligned}
\abs{\nabla\rhot\cdot\psi - \p_1  \rhot^r\psi_1} & \leq C\e e^{-\alpha t}\abs{\psi}, \\ 
\abs{ \left(\psi\dnab \uvt\right)\dnab\phi - \p_1 \ut_1^r \psi_1\p_1 \phi} & \leq C\e e^{-\alpha t}\abs{\psi} \abs{\nabla\phi}, \\
\abs{\nabla\rhot\dnab\phi-\p_1  \rhot^r \p_1 \phi} & \leq C \e e^{-\alpha t}\abs{\nabla \phi}.
\end{aligned} 
\end{equation*}
And combining the fact that $ \norm{\g}_{L^\infty(\Omega)} \leq C\e e^{-\alpha t}, $ $ \p_1 \rhot_1^r \leq C \p_1 \ut_1^r $ and $ \norm{\phi,\psi}_{L^\infty(\Omega)} \leq N(T) \leq \nu_0 <1, $
it holds that
\begin{align}
\int_0^T \norm{I_3}_{L^1(\Omega)} dt & \leq C N(T) \int_0^T \left(\norm{\p_1^2\ut_1^r}+ \norm{ \p_1^2 \rhot^r}\right) \norm{\nabla\phi}dt + C\e \int_0^T e^{-\alpha t} \norm{\phi,\psi} \norm{\nabla(\phi,\psi)}  dt \notag \\
& \quad + C \int_0^T \int_\Omega \p_1  \ut_1^r \left(\abs{\phi}+\abs{\psi_1}\right) \left(\abs{\nabla\phi}+\abs{\nabla\psi}\right) dxdt \notag \\
& \leq C\delta^{\frac{1}{2}} + C (\delta^{\frac{1}{2}}+\e) \int_0^T \norm{\nabla \left(\phi, \psi\right)}^2 dt +  C\e \sup_{t\in(0,T)} \norm{\phi,\psi}^2  \notag \\
& \quad + C \int_0^T \norm{\left(\p_1 \ut_1^r\right)^{\frac{1}{2}} (\phi,\psi_1)}^2 dt. \label{I-3}
\end{align}
For $ I_5, $ it can follow easily from \cref{Lem-h} and \cref{g-f} that
\begin{align}
\int_0^T \norm{I_5}_{L^1(\Omega)} dt \leq~& C\int_0^T \e e^{-\alpha t} \norm{\nabla(\phi, \psi)} dt + C \int_0^T \min\{\delta, (1+t)^{-1}\} \norm{\nabla\phi} dt \notag \\
\leq~& C(\e+\delta^{\frac{1}{2}}) \int_0^T \norm{\nabla\phi}^2 dt + C\e \int_0^T \norm{\nabla\psi}^2 dt + C(\e+\delta^{\frac{1}{2}}). \label{I-5}
\end{align}
And for the most difficult term $ I_4, $ decompose $ \psi = \sum\limits_{k=1}^3 \psi^{(k)} $ as in \cref{Lem-GN} and the a priori assumption \cref{apriori} yield that
\begin{align}
\int_0^T \norm{I_4}_{L^1(\Omega)} dt & \leq C \int_0^T \int_\Omega \abs{\nabla\psi} \abs{\nabla \phi}^2 dx \notag \\
& \leq C \sum_{k=1}^3 \int_0^T \norm{\nabla\psi^{(k)}}_{L^\infty(\Omega)} \norm{\nabla\phi}^2 dt \notag \\
& \leq C \int_0^T \left( \norm{\nabla^2 \psi }^{\frac{1}{2}} \norm{\nabla\psi}^{\frac{1}{2}} + \norm{\nabla^2 \psi } + \norm{\nabla^3 \psi }^{\frac{3}{4}} \norm{\nabla\psi}^{\frac{1}{4}} \right) \norm{\nabla\phi}^2 dt \notag \\
& \leq C \int_0^T \left( \norm{\nabla\psi}_1 + \norm{\nabla^3 \psi} \right)\norm{\nabla\phi}^2 dt \notag \\
& \leq C N(T) \int_0^T \norm{\nabla\phi}^2 dt + C N(T) \int_0^T \norm{\nabla^3 \psi} \norm{\nabla\phi} dt \notag \\
& \leq C \nu_0 \int_0^T \norm{\nabla\phi}^2 dt + C \nu_0 \int_0^T \norm{\nabla^3 \psi}^2 dt. \label{I-4}
\end{align}
Thus, by integrating \cref{eq-8} over $  \Omega\times (0,T), $ using \cref{I-2,I-3,I-4,I-5} and applying \cref{Lem-est-0}, one can finish the proof.

\end{proof}

\vspace{0.3cm}

\begin{Lem}\label{Lem-est-2}
Under the assumptions of \cref{Thm-apriori}, if $ \delta_0>0, \e_0>0 $ and $ \nu_0>0 $ are small enough, then
\begin{equation}\label{eq-est-2}
\sup_{t\in(0,T)} \norm{\nabla\psi}^2 + \int_0^T \norm{\nabla^2 \psi}^2 dt \leq C \norm{\phi_0, \psi_0}_1^2 + C(\e+\delta^{\frac{1}{4}}) + C \nu_0 \int_0^T \norm{\nabla^3 \psi}^2 dt.
\end{equation}	
\end{Lem}

\begin{proof}

Multiplying $ -\frac{\lap \psi}{\rho} \cdot $ on  \cref{eq-pertur-2} yields that
\begin{align*}
& -\p_t\psi \cdot\lap \psi - (\uv\dnab \psi) \cdot\lap \psi - (\psi \dnab\uvt) \cdot \lap\psi - \frac{p'(\rho)}{\rho} \nabla\phi\cdot\lap \psi  \\
& \qquad - \Big(\frac{p'(\rho)}{\rho} - \frac{p'(\rhot)}{\rhot}\Big) \nabla\rhot\cdot\lap\psi + \frac{\mu}{\rho} \abs{\lap\psi}^2 + \frac{\mu+\lambda}{\rho} \nabla\dv\psi\cdot\lap\psi \\
& \quad = - \frac{1}{\rho}\f \cdot\lap \psi + \frac{\phi}{\rho} \g\cdot \lap \psi - \frac{2\mu+\lambda}{\rho} \p_1 ^2\ut_1^r \lap\psi_1.
\end{align*}
Note that $ -\p_t\psi \cdot\lap \psi = \frac{1}{2} \p_t \left(\abs{\nabla\psi}^2\right) - \dv \left(\nabla\psi \p_t\psi\right) $ and
\setlength{\abovedisplayskip}{3pt}
\begin{align}
\frac{1}{\rho} \nabla\dv\psi\cdot\lap\psi =~& \frac{1}{\rho} \abs{\nabla\dv\psi}^2 + \frac{1}{\rho} \nabla\dv\psi\cdot(\lap\psi-\nabla\dv\psi) \notag \\
=~& \frac{1}{\rho} \abs{\nabla\dv\psi}^2 + \frac{1}{\rho} \dv\left[ \dv\psi \left(\lap\psi-\nabla\dv\psi \right) \right] \notag \\
=~& \frac{1}{\rho} \abs{\nabla\dv\psi}^2 + \dv\Big[\frac{\dv\psi}{\rho}  \left(\lap\psi-\nabla\dv\psi \right) \Big] \notag \\
& + \frac{\dv\psi}{\rho^2} (\nabla\rhot+\nabla\phi) \cdot \left( \lap \psi - \nabla \dv\psi \right). \label{hot}
\end{align}
Then one has that
\begin{equation}\label{eq-10}
\frac{1}{2} \p_t \abs{\nabla\psi}^2 + \frac{\mu}{\rho} \abs{\lap\psi}^2 + \frac{\mu+\lambda}{\rho} \abs{\nabla\dv\psi}^2 = \dv Q_4 + \sum\limits_{j=6}^9 I_j,
\end{equation}
where $ Q_4 = \nabla\psi \p_t\psi - \frac{\mu+\lambda}{\rho} \dv\psi \left(\lap\psi-\nabla\dv\psi \right) $ and
\begin{align*}
I_6 & = (\uv\dnab \psi) \cdot \lap \psi + \frac{p'(\rho)}{\rho} \nabla\phi\cdot\lap\psi-\frac{\mu+\lambda}{\rho^2} \dv\psi \nabla\rhot \cdot \left( \lap \psi - \nabla \dv\psi \right), \\
I_7 & = (\psi\dnab \uvt) \cdot\lap\psi + \Big(\frac{p'(\rho)}{\rho} - \frac{p'(\rhot)}{\rhot}\Big) \nabla\rhot\cdot\lap\psi + \frac{\phi}{\rho} \g\cdot\lap\psi , \\
I_8 & = - \frac{1}{\rho} \f\cdot \lap\psi - \frac{2\mu+\lambda}{\rho} \p_1 ^2 \ut_1^r \lap\psi_1, \\
I_9 & =  - \frac{\mu+\lambda}{\rho^2} \dv\psi \nabla\phi\cdot(\lap\psi-\nabla\dv\psi).
\end{align*}
Since $ \int_\Omega \frac{\mu}{\rho} \abs{\lap \psi}^2 dx \geq C \norm{\lap \psi}^2 \geq a_0 \norm{\nabla^2\psi}^2 $ for some $ C>0 $ and $ a_0>0, $ which depend only on $ \mu \big(\inf\limits_{x,t} \rho(x,t) \big)^{-1}, $
then similar to the proof of \cref{Lem-est-1}, one can get that
\begin{equation*}
\begin{aligned}
\int_0^T \norm{I_6}_{L^1(\Omega)} dt & \leq \frac{a_0}{8} \int_0^T \norm{\nabla^2 \psi}^2 dt + C \int_0^T \norm{\nabla\left(\phi,\psi\right)}^2dt, \\
\int_0^T \norm{I_7}_{L^1(\Omega)} dt & \leq C\e \sup_{t\in(0,T)} \norm{\phi,\psi}^2 + C(\e+\delta) \int_0^T \norm{\nabla^2 \psi}^2 dt + C \int_0^T \norm{\left(\p_1 \ut_1^r\right)^{\frac{1}{2}} (\phi,\psi_1)}^2 dt, \\
\int_0^T \norm{I_8}_{L^1(\Omega)} dt & \leq \frac{a_0}{8} \int_0^T \norm{\nabla^2 \psi}^2 dt + C (\e+\delta), \\
\int_0^T \norm{I_9}_{L^1(\Omega)} dt & \leq C \int_0^T \norm{\dv\psi}_{L^\infty(\Omega)} \norm{\nabla \phi}\norm{\nabla^2\psi} dt \\
& \leq C \int_0^T  \left( \norm{\nabla\psi}_1 + \norm{\nabla^3\psi} \right) \norm{\nabla \phi}\norm{\nabla^2\psi} dt \\
& \leq C N(T) \int_0^T \left( \norm{\nabla\psi}_1 \norm{\nabla^2\psi} + \norm{\nabla^2\psi} \norm{\nabla^3\psi} \right) dt \\
& \leq C \nu_0 \int_0^T \norm{\nabla \psi}_1^2 dt + C \nu_0 \int_0^T \norm{\nabla^3 \psi}^2 dt.
\end{aligned}
\end{equation*}
Then by integrating \cref{eq-10} over $ \Omega\times(0,T) $ and applying Lemmas \ref{Lem-est-0} and \ref{Lem-est-1}, one can finish the proof.

\end{proof}

\vspace{0.3cm}

\begin{Lem}\label{Lem-est-3}
Under the assumptions of \cref{Thm-apriori}, if $ \delta_0>0, \e_0>0 $ and $ \nu_0>0 $ are small enough, then 
\begin{equation}\label{eq-est-3}
\sup_{t\in(0,T)} \norm{\nabla^2\phi}^2 + \int_0^T \norm{\nabla^2 \phi}^2 dt \leq C \norm{\phi_0}_2^2 + \norm{\psi_0}_1^2 + C(\e+\delta^{\frac{1}{4}}) + C \nu_0 \int_0^T \norm{\nabla^3 \psi}^2 dt.
\end{equation}
\end{Lem}

\begin{proof}
Let $ i \in \{1,2,3\} $ be fixed.
Taking the second derivative $ \nabla \p_i  $ on \cref{eq-pertur-1}, then multiplying the result by $ \cdot \frac{\nabla \p_i  \phi}{\rho^2} $ and using the fact
\begin{align*}
- \frac{\p_t\rho}{\rho^3} \abs{\nabla \p_i  \phi}^2 + \frac{1}{2\rho^2} \left(\uv\dnab\right) \abs{\nabla \p_i  \phi}^2 = \dv\Big( \frac{\abs{\nabla \p_i  \phi}^2}{2\rho^2} \uv \Big) - \frac{3}{2\rho^2} \dv \uv \abs{\nabla \p_i  \phi}^2,
\end{align*}
one can get that
\begin{equation}\label{eq-11}
\begin{aligned}
& \p_t \Big(\frac{\abs{\nabla \p_i  \phi}^2}{2\rho^2}\Big)
+ \frac{1}{\rho} \nabla \p_i  \dv\psi \cdot \nabla \p_i  \phi = - \dv\Big( \frac{\abs{\nabla \p_i  \phi}^2}{2\rho^2} \uv \Big) + I_{10} + I_{11},
\end{aligned}
\end{equation}
where 
\begin{align}
I_{10} =~& \frac{3}{2\rho^2} \dv \uv \abs{\nabla \p_i  \phi}^2 - \frac{1}{\rho^2} \left[\nabla \p_i  (\rho \dv\psi) - \rho \nabla \p_i  \dv\psi \right] \cdot \nabla \p_i  \phi \notag \\ 
& - \frac{1}{\rho^2} \left[\nabla \p_i  (\uv\dnab\phi) - (\uv \dnab) \nabla \p_i  \phi\right] \cdot \nabla \p_i  \phi \notag \\
& - \frac{1}{\rho^2} \left[\nabla \p_i  (\phi \dv\uvt) - \phi \nabla\p_i  \dv\uvt \right] \cdot \nabla \p_i  \phi - \frac{1}{\rho^2} \left[\nabla \p_i  (\nabla \rhot \cdot \psi) - \nabla^2 \p_i  \rhot \psi \right] \cdot \nabla \p_i  \phi \notag \\
I_{11} =~& - \frac{\phi}{\rho^2} \nabla\p_i  \dv\uvt \cdot \nabla\p_i  \phi - \frac{1}{\rho^2} \nabla^2 \p_i  \rhot \psi \cdot \nabla \p_i  \phi - \frac{1}{\rho^2} \nabla \p_i  h_0 \cdot \nabla \p_i  \phi. \notag
\end{align}
Note that $ \frac{1}{\rho} \nabla\p_i  \dv\psi \dnab \p_i \phi $ is a high-order term in \cref{eq-11}, which can be canceled by the equation \cref{eq-pertur-2}. And similar to the proof of Lemmas \ref{Lem-est-0} to \ref{Lem-est-2}, the other lower-order terms satisfy that
\begin{align*}
\abs{I_{10}} \leq~& C \abs{\nabla^2\phi} \Big[ (\abs{\nabla\uvt}+\abs{\nabla\psi}) \abs{\nabla^2 \phi} + (\abs{\nabla^2\rhot} + \abs{\nabla^2 \phi}) \abs{\nabla\psi} + (\abs{\nabla \rhot} + \abs{\nabla \phi}) \abs{\nabla^2 \psi} \\
& \qquad \qquad  + (\abs{\nabla^2 \uvt} + \abs{\nabla^2\psi}) \abs{\nabla\phi} \Big] \\
\leq~& C (\e+\delta) \abs{\nabla^2\phi} \left( \abs{\nabla^2 \phi} + \abs{\nabla\psi} + \abs{\nabla^2\psi} + \abs{\nabla\phi} \right) + C \abs{\nabla\psi} \abs{\nabla^2\phi}^2 \\
& + C \abs{\nabla\phi} \abs{\nabla^2\psi} \abs{\nabla^2 \phi}, \\
\abs{I_{11}} \leq~& C \norm{\phi,\psi}_{L^\infty} \left(\abs{\p_1^2\p_i \ut_1^r} + \abs{\p_1^2\p_i \rhot_1^r} \right) \abs{\nabla^2 \phi} + C \e e^{-\alpha t} (\abs{\phi}+\abs{\psi}) \abs{\nabla^2 \phi} + C \abs{\nabla^2 h_0} \abs{\nabla^2\phi}.
\end{align*}
For $ I_{10}, $ first note that by \cref{Lem-GN}, one has
\begin{equation}\label{L4}
\begin{aligned}
\norm{\nabla \phi}_{L^4(\Omega)} & \leq C \sum\limits_{k=1}^3 \norm{\nabla^2 \phi}^{k/4} \norm{\nabla \phi}^{1-k/4} \leq C N(T), \\
\norm{\nabla^2\psi}_{L^4(\Omega)} & \leq C \sum\limits_{k=1}^3 \norm{\nabla^3 \psi}^{k/4} \norm{\nabla^2 \psi}^{1-k/4}.
\end{aligned}
\end{equation}
Then combining the a priori assumption \cref{apriori}, one has that 
\begin{align}
\int_0^T \int_\Omega \abs{\nabla\phi} \abs{\nabla^2\psi} \abs{\nabla^2 \phi} dxdt & \leq \int_0^T \norm{\nabla\phi}_{L^4(\Omega)} \norm{\nabla^2 \psi}_{L^4(\Omega)} \norm{\nabla^2 \phi} dt \notag \\ 
& \leq C N(T) \int_0^T \sum\limits_{k=1}^3  \norm{\nabla^3\psi}^{k/4} \norm{\nabla^2\psi}^{1-k/4} \norm{\nabla^2\phi} dt \notag \\
& \leq C \nu_0 \int_0^T \left( \norm{\nabla^3 \psi}^2 + \norm{\nabla^2 \psi}^2 +\norm{\nabla^2 \phi}^2  \right) dt \notag \\
& \leq C \nu_0 \int_0^T \norm{\nabla^2 (\phi,\psi)}^2 dt + C \nu_0 \int_0^T \norm{\nabla^3 \psi}^2 dt, \label{est-high}
\end{align}
where the Holder inequality for $ 1=\frac{k}{8} + \frac{4-k}{8} +\frac{1}{2} $ is used.
Thus, it holds that
\begin{align}
\int_0^T \norm{I_{10}}_{L^1(\Omega)} dt \leq~& C(\e+\delta) \int_0^T \norm{\nabla \left(\phi,\psi\right)}_1^2 dt + C \int_0^T \norm{\nabla\psi}_{L^\infty(\Omega)} \norm{\nabla^2\phi}^2 dt \notag \\
& + C \nu_0 \int_0^T \norm{\nabla^2 (\phi,\psi)}^2 dt + C \nu_0 \int_0^T \norm{\nabla^3 \psi}^2 dt \notag \\
\leq~& C(\e+\delta+\nu_0) \int_0^T \norm{\nabla \left(\phi,\psi\right)}_1^2 dt  + C \nu_0 \int_0^T \norm{\nabla^3 \psi}^2 dt \notag \\
& + C N(T) \int_0^T \left( \norm{\nabla\psi}_1 + \norm{\nabla^3\psi} \right) \norm{\nabla^2\phi} dt \notag \\
\leq~& C(\e+\delta+\nu_0) \int_0^T \norm{\nabla \left(\phi,\psi\right)}_1^2 dt  + C \nu_0 \int_0^T \norm{\nabla^3 \psi}^2 dt. \label{I-10}
\end{align}

For $ I_{11}, $ one has that 
\begin{align}
\int_0^T \norm{I_{11}}_{L^1(\Omega)} dt
\leq~& C N(T) \int_0^T \min\{\delta,(1+t)^{-1}\} \norm{\nabla^2 \phi} dt \notag \\
& + C \e \int_0^T  e^{-\alpha t} \norm{\phi,\psi} \norm{\nabla^2 \phi} dt + C\e \int_0^T e^{-\alpha t} \norm{\nabla^2 \phi} dt \notag\\
\leq~& C (\e+\nu_0) \int_0^T \norm{\nabla^2 \phi}^2 dt + C(\e+\delta) +  C \e \sup_{t\in(0,T)} \norm{\phi,\psi}^2 . \label{I-11}
\end{align}

\vspace{0.3cm}

For fixed $ i \in \{1,2,3\}, $ taking the derivative $ \p_i  $ on $ \frac{1}{\rho}\cdot$\cref{eq-pertur-2}, then multiplying the result by $ \cdot \nabla \p_i  \phi, $ and using the fact that
\begin{align*}
\p_t \p_i \psi \dnab \p_i \phi & = \p_t\left( \p_i \psi \cdot\nabla\p_i  \phi \right) - \dv\left( \p_i \p_t \phi \p_i \psi \right) + \p_i \p_t\phi \dv\p_i \psi \\
& = \p_t\left( \p_i \psi \cdot\nabla\p_i  \phi \right) - \dv\left( \p_i \p_t \phi \p_i \psi \right) \\
& \quad - \p_i  \left( h_0 + \rho \dv\psi + \uv\dnab\phi+ \dv\uvt\phi +\nabla\rhot\cdot\psi \right) \dv\p_i \psi, 
\end{align*}
and 
\begin{align*}
& \frac{1}{\rho} \left[ \mu \lap \p_i \psi + (\mu+\lambda) \nabla\dv \p_i \psi \right] \cdot \nabla \p_i \phi \\
&\quad = \frac{2\mu+\lambda}{\rho} \nabla\dv \p_i \psi \cdot \nabla\p_i \phi + \mu \dv \left(\frac{ \nabla \p_i \phi \times \text{curl} \p_i \psi }{\rho} \right) + \frac{\mu \nabla \rho \cdot \left(\nabla \p_i \phi \times \text{curl} \p_i \psi \right) }{\rho^2}
\end{align*}
(which is similar to \cref{eq-7}), one can get that
\begin{equation}\label{eq-14}
\p_t\left( \p_i \psi \cdot\nabla\p_i  \phi \right) + \frac{p'(\rho)}{\rho} \abs{\nabla\p_i  \phi}^2 - \frac{2\mu+\lambda}{\rho} \nabla \dv \p_i \psi \dnab\p_i  \phi = \dv Q_5 + \sum\limits_{j=12}^{14} I_j,
\end{equation}
where $ Q_5 = \p_i \p_t \phi \p_i \psi- \frac{\mu}{\rho} \nabla\p_i  \phi\times \text{curl}\p_i  \psi $ and
\begin{align*}
I_{12} &= \p_i  \left( \frac{\f-\phi\g}{\rho} \right) \dnab\p_i  \phi + \p_i  \left(\frac{2\mu+\lambda}{\rho} \p_1 ^2\ut_1^r \right) \p_{1i} \phi + \p_i  h_0 \dv\p_i \psi, \\
I_{13} &= \p_i  \left( \dv\uvt\phi +\nabla\rhot\cdot\psi \right) \dv\p_i \psi + \p_i  \left( \psi\dnab\uvt \right)\cdot\nabla\p_i \phi \\
& \quad + \p_i  \left[ \left(\frac{p'(\rho)}{\rho} - \frac{p'(\rhot)}{\rhot}\right) \nabla\rhot \right] \cdot \nabla\p_i \phi, \\
I_{14} &= \p_i  \left( \rho \dv\psi + \uv\dnab\phi \right) \dv\p_i \psi + \p_i  \left(\uv\dnab\psi \right)\cdot\nabla\p_i \phi + \p_i  \left(\frac{p'(\rho)}{\rho}\right) \nabla\phi\cdot\nabla\p_i \phi \\
& \quad - \p_i  \left(\frac{\mu}{\rho}\right)\lap\psi \dnab\p_i  \phi - \p_i  \left(\frac{\mu+\lambda}{\rho}\right) \nabla\dv\psi \dnab\p_i  \phi  + \frac{\mu \nabla \rho \cdot \left(\nabla \p_i \phi \times \text{curl} \p_i \psi \right) }{\rho^2}.
\end{align*}
Similar to the estimates of \cref{I-10,I-11}, one can get that
\begin{align*}
\int_0^T \norm{I_{12}}_{L^1(\Omega)} dt \leq~& C \int_0^T \left(\norm{\nabla(\f-\phi\g)} +  \norm{\p_1 ^2\p_i \ut_1^r}  \right)\norm{\nabla^2\phi} dt \\
& + C \int_0^T \int_\Omega \left(\abs{\f-\phi \g} + \abs{\p_1^2 \ut_1^r} \right) (\e+\delta+\abs{\nabla \phi}) \abs{\nabla^2\phi} dx dt \\
&  + \int_0^T \norm{\nabla h_0} \norm{\nabla^2\psi} dt \\
\leq~& C \int_0^T \left(\norm{\f-\phi\g}_1 +  \norm{\p_1 ^2 \ut_1^r}_1  \right)\norm{\nabla^2\phi} dt \\
& + C \int_0^T \left(\norm{\f-\phi \g}_{L^\infty} + \delta \right)\norm{\nabla \phi} \norm{\nabla^2\phi} dt + C\e \int_0^T e^{-\alpha t} \norm{\nabla^2\psi} dt \\
\leq~& C \int_0^T \left( \e e^{-\alpha t} + \e e^{-\alpha t} N(T) + \min\{\delta,(1+t)^{-1}\}  \right)\norm{\nabla^2\phi} dt \\
& + C \int_0^T \left(\e + \e N(T) + \delta\right) \norm{\nabla\phi} \norm{\nabla^2\phi} dt + C\e + C\e \int_0^T \norm{\nabla^2\psi}^2 dt \\
\leq~& C(\e+\delta^{\frac{1}{2}}) \int_0^T \norm{\nabla^2\phi}^2 dt + C(\e+\delta^{\frac{1}{2}}) + C \int_0^T \norm{\nabla\phi,\nabla^2\psi}^2 dt;\\
\int_0^T \norm{I_{13}}_{L^1(\Omega)} dt \leq~& C\int_0^T \int_\Omega \left( \abs{\p_1^2 \ut_1^r} + \abs{\p_1^2 \rhot^r} + \abs{\p_1 \rhot^r}^2 +\e e^{-\alpha t} \right) \\
& \qquad \qquad \times (\abs{\phi}+\abs{\psi}) (\abs{\nabla^2\psi} + \abs{\nabla^2 \phi}) dxdt \\
& + C(\e+\delta) \int_0^T \norm{\nabla(\phi,\psi)} \norm{\nabla^2(\psi, \phi)} dt \\
\leq~& C N(T) \int_0^T \left(\min\{\delta,(1+t)^{-1}\} + \min\{\delta^2,\delta^{\frac{1}{2}} (1+t)^{-\frac{3}{2}}\} \right) \norm{\nabla^2(\phi,\psi)} dt \\
&  + C\e \sup_{t\in(0,T)} \norm{\phi,\psi}^2 + C(\e+\delta)\int_0^T \norm{\nabla^2(\psi,\phi)}^2 dt + C\int_0^T \norm{\nabla(\phi,\psi)}^2 dt \\
\leq~& C (\e+\delta+\nu_0) \int_0^T \norm{\nabla^2 \phi}^2 dt + C \delta + C\e \sup_{t\in(0,T)} \norm{\phi,\psi}^2  \\
& + C \int_0^T (\norm{\nabla\phi}^2 + \norm{\nabla\psi}_1^2) dt; \\
\int_0^T \norm{I_{14}}_{L^1(\Omega)} dt \leq~& C \int_0^T \int_\Omega \Big[ \abs{\nabla^2\psi} \left( \abs{\nabla^2\psi}+\abs{\nabla^2 \phi} +\abs{\nabla\psi} + \abs{\nabla \phi} \right) \\
& \qquad \qquad + \abs{\nabla^2\phi} \left(\abs{\nabla\psi}+\abs{\nabla\phi}\right) \Big] dx dt \\
& + C \int_0^T \int_\Omega \Big[  \abs{\nabla^2\psi} \left(  \abs{\nabla\phi} \abs{\nabla^2\phi} + \abs{\nabla\phi} \abs{\nabla\psi} \right) \\
& \qquad\qquad + \abs{\nabla^2\phi} \left( \abs{\nabla\psi}^2 + \abs{\nabla\phi}^2 \right) \Big] dxdt.
\end{align*}
Besides $ \abs{\nabla\phi}\abs{\nabla^2\psi}\abs{\nabla^2\phi} $ which has been estimated in \cref{est-high}, the other triple terms in $ I_{14} $ can be estimated as follows.
\begin{align*}
& \int_0^T \int_\Omega \abs{\nabla\psi} \left(\abs{\nabla^2\psi}\abs{\nabla\phi} + \abs{\nabla^2\phi} \abs{\nabla\psi}\right)  dx dt + \int_0^T \int_\Omega \abs{\nabla\phi}^2 \abs{\nabla^2\phi} dx dt \\
& \quad \leq C \int_0^T  \norm{\nabla\psi}_{L^\infty}\left( \norm{\nabla\phi}\norm{\nabla^2\psi} + \norm{\nabla\psi}\norm{\nabla^2\phi} \right) dt + C \int_0^T \norm{\nabla\phi}_{L^4(\Omega)}^2 \norm{\nabla^2 \phi} dt \\
& \quad \leq C N(T) \int_0^T \left(\norm{\nabla\psi}_1 + \norm{\nabla^3\psi}\right)\left( \norm{\nabla^2\psi} + \norm{\nabla^2\phi} \right) dt \\
& \qquad + C \int_0^T \sum\limits_{k=1}^3 \norm{\nabla^2\phi}^{1+k/2} \norm{\nabla\phi}^{2-k/2} dt \\
& \quad \leq C \nu_0 \int_0^T \left(\norm{\nabla\psi}_1^2 + \norm{\nabla^2 \phi}^2 \right) dt + C \nu_0 \int_0^T \norm{\nabla^3 \psi}^2 dt \\
& \qquad + C N(T) \int_0^T \sum\limits_{k=1}^3 \norm{\nabla^2\phi}^{k/2} \norm{\nabla\phi}^{2-k/2} dt, \\
& \quad \leq C \nu_0 \int_0^T \norm{\nabla(\psi, \phi)}_1^2 dt + C \nu_0 \int_0^T \norm{\nabla^3 \psi}^2 dt.
\end{align*}
Thus, it holds that
\begin{align*}
\int_0^T \norm{I_{14}}_{L^1(\Omega)} dt \leq~& \left(\frac{1}{8}+C\nu_0 \right) \int_0^T \int_\Omega \frac{p'(\rho)}{\rho} \abs{\nabla^2\phi}^2 dxdt + C \int_0^T \left( \norm{\nabla\psi}_1^2 + \norm{\nabla\phi}^2\right) dt \\
& + C \nu_0 \int_0^T \norm{\nabla^3 \psi}^2 dt.
\end{align*}
Thus, collecting the estimates of $ I_{10} $ to $ I_{14} $ above, by adding the two equations $ (2\mu+\lambda)\cdot $\cref{eq-11} and \cref{eq-14} together, and summing the results with respect to $ i $ from $ 1 $ to $ 3, $ one has that
\begin{align*}
\sup_{t\in(0,T)} \norm{\nabla^2\phi}^2 + \int_0^T \norm{\nabla^2 \phi}^2 dt \leq~& C \norm{\phi_0}_2^2 + C \norm{\psi_0}_1^2 + C \norm{\nabla\psi}^2 + C (\e+\delta^{\frac{1}{2}}) \\
& + C \int_0^T \left(\norm{\nabla\psi}_1^2 + \norm{\nabla\phi}^2 \right) dt + C \nu_0 \int_0^T \norm{\nabla^3 \psi}^2 dt.
\end{align*}
Thus \cref{eq-est-3} follows from Lemmas \ref{Lem-est-0}--\ref{Lem-est-2}.
\end{proof}

\begin{Lem}\label{Lem-est-4}
Under the assumptions of \cref{Thm-apriori}, if $ \delta_0>0, \e_0>0 $ and $ \nu_0>0 $ are small enough, then
\begin{equation}\label{eq-est-4}
\sup_{t\in(0,T)} \norm{\nabla^2 \psi}^2 + \int_0^T \norm{\nabla^3 \psi}^2 dt \leq C \norm{\phi_0, \psi_0}_2^2 + C(\e+\delta^{\frac{1}{4}}).
\end{equation}	
\end{Lem}

\begin{proof}
For fixed $ i \in \{1,2,3\}, $ taking the derivative $ \p_i  $ on $ \frac{1}{\rho}\cdot $\cref{eq-pertur-2} and then multiplying the result by $ \cdot(- \p_i  \lap \psi), $ one has that
\begin{align}
& -\p_t\p_i \psi \cdot \p_i \lap\psi + \p_i \left( \frac{\mu}{\rho} \lap \psi \right)  \cdot\p_i \lap\psi + \p_i \left( \frac{\mu+\lambda}{\rho} \nabla\dv \psi \right) \cdot\p_i \lap\psi \notag \\
& \quad = \p_i \lap\psi \cdot \p_i  \Big[\uv\dnab\psi 
+ \psi\dnab\uvt + \frac{p'(\rho)}{\rho} \nabla\phi + \left( \frac{p'(\rho)}{\rho} - \frac{p'(\rhot)}{\rhot} \right) \nabla\rhot \notag \\
& \qquad\qquad\qquad\quad -\frac{\f-\phi\g}{\rho} - \frac{2\mu+\lambda}{\rho} \p_1 ^2\ut_1^r \E_1 \Big] := I_{15}. \label{eq-16}
\end{align}
For the left-hand side of \cref{eq-16}, similar to \cref{hot}, one can get that
\begin{align*}
	-\p_t\p_i \psi \cdot \p_i \lap\psi & =
	\frac{1}{2} \p_t \left( \abs{\nabla\p_i  \psi}^2 \right) - \dv\left( \nabla\p_i  \psi \p_t \p_i  \psi \right), \\
	\frac{1}{\rho} \p_i  \nabla\dv \psi \cdot\p_i \lap\psi & = \frac{1}{\rho} \abs{\nabla\dv \p_i \psi}^2 + \dv \left( \frac{\dv\p_i \psi}{\rho} \lap\p_i \psi - \frac{\nabla \left(\dv\p_i \psi\right)^2}{2\rho} \right) \\
	&\quad + \frac{\dv\p_i \psi}{\rho} \nabla\rho \cdot \left( \lap\p_i  \psi - \nabla\p_i  \dv\psi \right).
\end{align*} 
Then one has that 
\begin{equation}\label{eq-15}
	\frac{1}{2} \p_t \left( \abs{\nabla\p_i  \psi}^2 \right) + \frac{\mu}{\rho} \abs{\p_i \lap\psi}^2 + \frac{1}{\rho} \abs{\nabla\dv \p_i \psi}^2 = \dv Q_6 + I_{15} + I_{16},
\end{equation}
where $ Q_6 = \nabla\p_i  \psi \p_t \p_i  \psi - \frac{\dv\p_i \psi}{\rho} \lap\p_i \psi + \frac{\nabla \left(\dv\p_i \psi\right)^2}{2\rho}, $
\begin{equation}\label{eq-I-16}
I_{16} = \frac{\p_i \rho}{\rho^2} \left[ \mu \lap\psi + (\mu+\lambda) \nabla\dv\psi \right] \cdot \p_i \lap\psi - \frac{\dv\p_i \psi}{\rho} \nabla\rho \cdot \left( \lap\p_i  \psi - \nabla\p_i  \dv\psi \right),
\end{equation}
and there holds that $ \int_\Omega \frac{\mu}{\rho} \abs{\p_i  \lap \psi}^2 dx \geq a_0 \norm{\nabla^2 \p_i  \psi}^2 $ for some constant $ a_0>0. $

From \cref{eq-16}, it holds that 
\begin{align}
\abs{I_{15}} & \leq C \abs{\p_i \lap\psi} \Big[ \abs{\nabla^2\psi} + \abs{\nabla\uvt}\abs{\nabla\psi} + \abs{\nabla\psi}^2 + \abs{\nabla^2\uvt}\abs{\psi} + \abs{\nabla^2\phi} + \abs{\nabla\rhot}\abs{\nabla \phi} + \abs{\nabla\phi}^2 \notag \\
& \qquad\qquad\quad + \left(\abs{\nabla \rhot}^2 + \abs{\nabla^2\rhot} \right)\abs{\phi} + \abs{\nabla\rhot} \abs{\nabla\phi} + \abs{\nabla\f} + \abs{\f} (\abs{\nabla\rhot}+\abs{\nabla\phi}) + \abs{\nabla\g} \abs{\phi}
\notag \\
& \qquad\qquad\quad + \abs{\g} \abs{\nabla\phi} + \abs{\g} \abs{\phi} (\abs{\nabla\rhot} + \abs{\nabla\phi}) +\abs{\p_1^2\p_i\ut_1^r} + \abs{\p_1 ^2\ut_1^r} (\abs{\nabla\rhot}+\abs{\nabla\phi}) \Big]. \notag
\end{align}
Similar to the estimates before, e.g.
\begin{align*}
& \int_0^T \int_\Omega \abs{\p_i \lap\psi} \left(\abs{\nabla \rhot}^2 + \abs{\nabla^2\rhot} \right)\abs{\phi} dx dt \\
&\quad \leq C N(T) \int_0^T \left( \min\{\delta^2, \delta^{\frac{1}{2}}(1+t)^{-\frac{3}{2}}\} +\min\{\delta,(1+t)^{-1}\} + C\e e^{-\alpha t} \right) \norm{\nabla^3\psi} dt \\
&\quad \leq (\delta + \e) + C \nu_0 \int_0^T \norm{\nabla^3 \psi}^2 dt,
\end{align*}
one can get that
\begin{align}
\int_0^T \norm{I_{15}}_{L^1(\Omega)} dt &  \leq C \int_0^T \norm{\nabla^3 \psi} \norm{\nabla(\psi,\phi)}_1 dxdt +  C \int_0^T \norm{\nabla^3 \psi} \norm{\nabla(\psi,\phi)}_{L^4(\Omega)}^2 dt \notag \\
& \quad + \int_0^T \norm{\nabla^3\psi} \left( \norm{\f}_1 + \norm{\p_1^2 \ut_1^r}_1 + \norm{\g}_{W^{1,\infty}} \norm{\phi} \right) dt \notag \\
& \quad + C (\delta + \e) + C \nu_0 \int_0^T \norm{\nabla^3 \psi}^2 dt \notag \\
& \leq \left(\frac{a_0}{4}+C (\e+\delta+\nu_0) \right) \int_0^T \norm{\nabla^3 \psi}^2 dt + C \int_0^T \norm{\nabla(\psi,\phi)}_1^2 dt \notag \\
& \quad  + C \e \sup_{t\in(0,T)} \norm{\phi}^2 + C (\delta + \e), \label{I-15}
\end{align}
where the $ L^4 $- estimate \cref{L4} and the Holder inequality for $ 1= \frac{1}{2}+\frac{k}{8} + \frac{4-k}{8} $ are used when dealing with the integral involving $ \norm{\psi,\phi}_{L^4(\Omega)}. $

Similarly, $ I_{16} $ satisfies that
\begin{align}
\int_0^T \norm{I_{16}}_{L^1(\Omega)} dt & \leq C \int_0^T \norm{\nabla^2\psi} \norm{\nabla^3\psi} dt + C \int_0^T  \norm{\nabla\phi}_{L^\infty} \norm{\nabla^2\psi} \norm{\nabla^3\psi} dt \notag \\
& \leq \left(\frac{a_0}{4}+C \nu_0 \right) \int_0^T \norm{\nabla^3 \psi}^2 dt + C \int_0^T \norm{\nabla^2 \psi}^2 dt. \label{I-16}
\end{align}
Then integrating \cref{eq-15} over $ \Omega \times(0,T) $ and summing the results with respect to $ i $ from $ 1 $ to $ 3 $ can finish the proof of \cref{Lem-est-4}.

\end{proof}

Thus, \cref{Thm-apriori} can follow from Lemmas \ref{Lem-est-0}--\ref{Lem-est-4} easily.

\vspace{0.3cm}

%
%
%
%
%
\begin{flushleft}
	\bfseries Proof of the main result, \cref{Thm}.
\end{flushleft}

\begin{proof}
	
	We first introduce the local existence theorem.
	\begin{Prop}[\cite{LWW}Local existence theorem]\label{Prop-local}
		There exists a positive constant $ b $, such that if $ \norm{\phi_0, \psi_0}_2\leq M $ and $ \inf\limits_{x}\left(\rhot(x,0)+\phi_0(x)\right) \geq m>0, $ then there exists a $T_0=T_0(m,M)$, such that, the problem \cref{eq-pertur-1}, \cref{eq-pertur-2} admits a unique solution $ (\phi, \psi)\in X_{\frac{1}{2}m, b M}(0,T_0) $ 
		where 
		\begin{align*}
		X_{m, M}(0,T) := \big\{ & (\phi,\psi): (\phi,\psi) \in C\left(0,T;H^2(\Omega)\right) \text{ with } ~ \inf_{x,t} \phi\geq m,~ \sup_{t} \norm{\phi,\psi}_2 \leq M, \\ 
		& \qquad \qquad \nabla\phi \in L^2\left(0,T; H^1(\Omega)\right) \text{ and }~ \nabla\psi\in L^2\left(0,T; H^2(\Omega)\right) \big\}.	
		\end{align*}
	\end{Prop}
	Under the assumptions of \cref{Thm}, we let $ m = \frac{1}{2} \rhob^- $ and $ M = C C_0 \e_0 + C (\e_0 + \delta_0^{\frac{1}{4}}), $ where $ C_0 $ is the constant in \cref{psi0} and $ C, \e_0,\delta_0 $ are the constants in \cref{Thm-apriori}. And we can let $ \e_0 $ and $ \delta_0 $ small enough such that
	\begin{equation}
	\inf_{x,t} \rhot(x,t) \geq \rhob^- - C\e_0 \geq \frac{1}{2} \rhob^- + M.
	\end{equation}
	
	Then by \cref{Prop-local}, the solution $ (\phi, \psi) $ to \cref{eq-pertur-1,eq-pertur-2} exists on $[0,T_0]$, satisfying $ (\phi, \psi)\in X_{\frac{1}{4}\rhob^-, bM}(0,T_0), $ i.e.
	$$\sup\limits_{0\leq t\leq T_0}\|(\phi, \psi)(t)\|_2\leq b M.$$
	Since either $ b $ or $ M $ is independent of the constant $ \nu_0>0 $ in \cref{Thm-apriori}, one can choose $ \e_0 $ and $ \delta_0 $ small enough such that $ b M < \nu_0. $
	Thus, it follows from the priori estimates, \cref{Thm-apriori}, that
	\begin{align*}
	\sup\limits_{0\leq t\leq T_0}\|(\phi, \psi)(t)\|_2\leq C \norm{\phi_0,\psi_0}_2 + C(\e+\delta^{\frac{1}{4}}) \leq M,
	\end{align*}
	and hence,
	\begin{align*}
	\inf_{\substack{x\in\Omega\\t\in(0,T_0)}} \left(\rhot(x,t) + \phi(x,t)\right) \geq \frac{1}{2} \rhob^- + M - M = \frac{1}{2} \rhob^-.
	\end{align*}
 	Using \cref{Prop-local} once more, the solution to \cref{eq-pertur-1,eq-pertur-2} also exists on$[T_0,2T_0]. $ By induction, we obtain the global in time solution to the Cauchy problem of \cref{eq-pertur-1,eq-pertur-2}, if $ \e_0>0 $ and $ \delta_0 $ are small enough. 
 	
 	Hence, to complete the proof of \cref{Thm}, it remains to prove the large time behavior of the solution.
 	
 	We give only the proof of $ \phi, $ since it is similar to prove $ \psi. $ It follows from \cref{Lem-GN} that
	\begin{align*}
	\norm{\phi}_{L^\infty(\R^3)} = \norm{\phi}_{L^\infty(\Omega)} & \leq C \norm{\nabla \phi}^{\frac{1}{2}} \norm{\phi}^{\frac{1}{2}} + \norm{\nabla\phi} + \norm{\nabla^2\phi}^{\frac{1}{2}} \norm{\nabla\phi}^{\frac{1}{2}}, \\
	& \leq C \delta^{\frac{1}{2}} \norm{\nabla\phi}^{\frac{1}{2}}.
	\end{align*}
	Thus if suffices to prove that $
	\norm{\nabla\phi}(t) \rightarrow 0 \quad \text{as } t\rightarrow +\infty, $ which can be easily derived from
	\begin{align*}
	\int_0^{+\infty} \left( \norm{\nabla \phi}^2 + \abs{\frac{d}{dt} \norm{\nabla \phi}^2 } \right) dt <+\infty.
	\end{align*}
	The proof of \cref{Thm} is completed.
	
\end{proof}

\vspace{0.3cm}

\section{Appendix}\label{Sec-appen}

\begin{proof}[Proof of \cref{Lem-per}]
	In this proof, we let $ \norm{\cdot} = \norm{\cdot}_{L^2(\Torus^3)} $ and $ \norm{\cdot}_l = \norm{\cdot}_{H^l(\Torus^3)} $ for $ l\geq 1. $
	Assume that the initial data to \cref{NS} is 
	$$ \left(\rho, \rho \uv\right)(x,0) = \left(\rhob, \rhob~ \uvb\right) + \left(v_0, \wv_0\right), \quad x\in\Torus^3, $$ 
	where $ \rhob $ and $ \uvb $ are constants with $ \rhob>0, $ and $ v_0 $ and $ \wv_0 $ are periodic functions on zero averages, i.e.
	\begin{align*}
	\int_{\Torus^3} \left(v_0, \wv_0\right) dx =0.
	\end{align*}
	By the Galilean transformation, one can assume that $ \uvb = 0 $ without loss of generality.
	If the initial data is small enough, i.e. $ \e = \norm{v_0,\wv_0}_5 << 1, $ then the global existence and uniqueness of the periodic solution $ \left(\rho,\uv\right) $ to \cref{NS} is standard; see \cite{MN1980}. And it holds that 
	\begin{align*}
	\sup_t \norm{(\rho-\rhob, \uv)}_5 \leq C\e.
	\end{align*}
	Now we prove the exponential decay rate of the solution by using the Poincar\'{e} inequality.
	Denote the perturbations
	\begin{align*}
	v := \rho-\rhob, \quad \zv := \uv-\uvb = \uv \quad \text{ and } \quad \wv := \rho \uv - \rhob ~\uvb = \rho \zv.
	\end{align*}
	Due to the conservative form of \cref{NS}, it holds that
	\begin{align*}
	\int_{\Torus^3} v(x,t) = 0, \quad \int_{\Torus^3} \wv(x,t) dx = 0, \quad t\geq 0.
	\end{align*}
	Thus, it follows from the Poincar\'{e} inequality that
	\begin{equation}\label{poinc-1}
	\begin{aligned}
	\norm{v} & \leq C_0 \norm{\nabla v}, \\
	\norm{\wv} & \leq C_0 \norm{\nabla \wv} \leq C_0 \norm{\zv}_{L^\infty} \norm{\nabla v} + C_1 \norm{\nabla \zv} \leq C_0\e \norm{\nabla v} + C_1 \norm{\nabla \zv},
	\end{aligned}
	\end{equation}
	where the positive constants $ C_0 $ and $ C_1 $ are independent of $ \e $ or $ t. $ For the derivatives, it also follows form the Poincar\'{e} inequality that
	\begin{equation}\label{poinc-2}
	\norm{\nabla v} \leq C_0 \norm{\nabla^2 v}, \quad \norm{\nabla \zv} \leq C_0\norm{\nabla^2 \zv} \quad \text{and } \norm{\nabla^2 \zv}\leq C_0 \norm{\nabla^3 \zv}.
	\end{equation}
	It is noted that the perturbations $ v $ and $ \zv $ here satisfy the same equations as \cref{eq-pertur-1,eq-pertur-2}, where $ \rhot$ and $\ut $ become constants and $ h_0=\f=\g =0. $ Thus, similar estimates as Lemmas \ref{Lem-est-0}--\ref{Lem-est-4} yield that
	\begin{equation}\label{app-ineq}
	E'(t) + C_2 \norm{\nabla v}_1^2 + C_2 \norm{\nabla \zv}_2^2 \leq 0,
	\end{equation}
	for some energy functional $ E(t) $ satisfying
	\begin{equation*}
	C^{-1} \norm{v,\zv}_2^2 \leq E(t) \leq C \norm{v,\zv}_2^2,
	\end{equation*}
	for some constant $ C>0. $
	Thus, it follows from \cref{poinc-1,poinc-2,app-ineq} that
	\begin{align*}
	E'(t) + \frac{C_2}{2} \norm{\nabla v}_1^2 + \frac{C_2C_0}{2} \norm{v}_1^2 + \frac{C_2}{2} \norm{\nabla \zv}_2^2 + C_3 \norm{\zv}_2^2 - C_4 \delta^2 \norm{\nabla v}^2 \leq 0.
	\end{align*}
	Thus, if $ \e>0 $ is small enough, one can get that
	\begin{equation*}
	E'(t) + 2 C_5 E(t) \leq 0,
	\end{equation*}
	which implies that
	\begin{equation*}
	\norm{v,\zv}_2(t) \leq Ce^{-C_5 t}.
	\end{equation*}
	And the higher order estimates 
	\begin{equation*}
	\norm{v,\zv}_5(t) \leq Ce^{-C_6 t}.
	\end{equation*}
	can be proved similarly, which is omitted for brevity.
	
\end{proof}

\vspace{0.3cm}

\begin{proof}[Proof of \cref{Lem-h}]
	The idea is to extract the ``well-decay terms'' $ R_i~ (i=1,2,\cdots) $ from the equations \cref{source}, where all the $ R_i $ are products of space-periodic functions decaying exponentially fast with respect to $ t $ (e.g. $ v^\pm, \zv^\pm, \rhot-\rhot^r, \p_t \uv^\pm, \nabla\rho^\pm, \cdots $) and integrable functions with respect to $ x_1 \in \R $ (e.g. $ \eta(1-\sigma), \sigma-\eta, \p_t \sigma, \p_1  \eta, \cdots $). 
	
	\vspace{0.2cm}
	
	Denote $ \e = \norm{v_0,\zv_0}_{H^5(\Torus^3)}. $ Then it follows from \cref{Lem-per} that $$ \norm{v^\pm,\zv^\pm}_{W^{3,+\infty}(\Torus^3)} \leq C\e e^{-2\alpha t}. $$

	i) For $ h_0 $ given in \cref{source}, first note that
	\begin{align}
	\p_t\rhot = ~ & \p_t \rho^- (1-\sigma) + \p_t \rho^+ \sigma + \left(\rho^+-\rho^-\right) \p_t \sigma, \notag \\
	= & \p_t \rho^- (1-\sigma) + \p_t \rho^+ \sigma + \p_t \rhot^r + R_1, \label{h0-1}
	\end{align}
	where the remainder $ R_1 = \left( v^+ -v^- \right) \p_t \sigma. $  It follows from Lemmas \ref{Lem-sig-eta} and \ref{Lem-per} that
	\begin{align*}
	\norm{R_1(\cdot,t)}_{W^{2,p}(\Omega)} & \leq C \e e^{-2 \alpha t} \norm{\p_t \sigma(\cdot,t)}_{W^{2,p}(\R)} \leq C \e e^{-2 \alpha t}.
	\end{align*}
	And
	\begin{align}
	\dv \left( \rhot \uvt  \right) = ~&  \left[\rho^-(1-\sigma) + \rho^+ \sigma\right] \left[ \dv \uv^- (1-\eta) + \dv \uv^+ \eta  \right] + \rhot \left[\p_1 \ut_1^r + \left(z_1^+ - z_1^-\right) \p_1  \eta\right] \notag \\
	& + \left[\uv^-(1-\eta) +\uv^+ \eta \right] \cdot \left[ \nabla \rho^- (1-\sigma) + \nabla \rho^+ \sigma  \right] + \ut_1 \left[\p_1 \rhot^r + \left(v^+ -v^- \right) \p_1  \sigma\right] \notag \\
	=~ & \dv \left(\rho^- \uv^-\right) (1-\sigma) (1-\eta) + \dv \left(\rho^+ \uv^+ \right) \sigma\eta + \p_1  \left(\rhot \ut_1^r\right) + R_2 \notag \\
	=~ & \dv \left(\rho^- \uv^-\right) (1-\sigma) + \dv \left(\rho^+ \uv^+ \right) \sigma + \p_1  \left(\rhot^r \ut_1^r \right) + R_2 + R_3, 
	\label{h0-2}
	\end{align}
	where 
	\begin{align*}
	R_2 =~ & \dv \left( \rho^- \uv^+\right) (1-\sigma)\eta + \dv \left( \rho^+ \uv^-  \right) \sigma (1-\eta)  + (\rhot-\rhot^r) \p_1  \ut_1^r + (\ut_1-\ut_1^r) \p_1 \rhot^r \\
	& + \rhot \left(z_1^+ - z_1^-\right) \p_1  \eta + \ut_1 \left(v^+ -v^-\right) \p_1  \sigma, \\
	R_3 =~ & - \dv \left(\rho^- \uv^-\right)(1-\sigma)\eta - \dv \left(\rho^+ \uv^+ \right) \sigma (1-\eta),
	\end{align*}
	which satisfy that
	\begin{align*}
	\sum_{i=2}^3 \norm{R_i}_{W^{2,p}(\Omega)} & \leq C \e e^{- \alpha t}.
	\end{align*}
	Collecting \cref{h0-1,h0-2}, one can get that
	\begin{align*}
	h_0 =~& \left[\p_t\rho^- + \dv(\rho^- \uv^-)\right] (1-\sigma) + \left[\p_t\rho^+ + \dv(\rho^+ \uv^+)\right] \sigma + \p_t \rhot^r + \p_1  \left(\rhot^r\ut^r_1\right)+ \sum\limits_{i=1}^3 R_i \\
	=~& \sum\limits_{i=1}^3 R_i,
	\end{align*}
	which satisfies
	\begin{equation*}
	\norm{h_0}_{W^{2,p}(\Omega)} \leq C \e e^{- \alpha t}.
	\end{equation*}
	
	\vspace{0.2cm}
	
	ii) Now we prove the source term $ \h $ deduced from the momentum equations \cref{source}. Note that
	\begin{align}
	\rhot \p_t \uvt =~ & \rhot \left[\p_t \uv^- (1-\eta) + \p_t \uv^+ \eta + \p_t \ut_1^r \E_1 + \left(\zv^+ - \zv^- \right)\p_t \eta \right] \notag \\
	=~ & \rho^- \p_t \uv^- (1-\sigma)(1-\eta) + \rho^+ \p_t \uv^+ \sigma\eta + \rhot \p_t \ut_1^r \E_1 + R_4 \notag \\
	=~ & \rho^- \p_t \uv^- (1-\sigma) + \rho^+ \p_t \uv^+ \sigma + \rhot^r \p_t \ut_1^r \E_1 + R_4 +R_5, \label{h1-1}
	\end{align}
	where
	\begin{align*}
	R_4 =~& \rho^- \p_t \uv^+ (1-\sigma) \eta + \rho^+ \p_t \uv^- \sigma(1-\eta) + \rhot \left( \zv^+ -\zv^-\right) \p_t \eta, \\
	R_5 =~& - \rho^- \p_t \uv^- (1-\sigma)\eta  - \rho^+ \p_t \uv^+ \sigma(1-\eta) + \left(\rhot-\rhot^r\right) \p_t \ut_1^r \E_1.
	\end{align*}
	And 	
	\begin{align}
	\rhot \uvt \dnab \uvt =~ & \rhot \uvt \cdot \left[\nabla \uv^- (1-\eta) + \nabla \uv^+ \eta \right] + \rhot \ut_1  \left[\p_1  \ut_1^r \E_1 + \p_1  \eta \left(\zv^+-\zv^- \right)\right] \notag \\
	=~ & \rho^- \uv^- \dnab \uv^- (1-\sigma)(1-\eta)^2 + \rho^+ \uv^+ \dnab \uv^+ \sigma \eta^2 + \rhot \ut_1 \p_1  \ut_1^r \E_1 + R_6 \notag \\
	=~ & \rho^- \uv^- \dnab  \uv^- (1-\sigma) + \rho^+ \uv^+ \dnab \uv^+ \sigma + \rhot^r \ut_1^r \p_1  \ut_1^r \E_1 + R_6 + R_7, \label{h1-2}
	\end{align}
	where
	\begin{align*}
	R_6 =~& \nabla \uv^- \cdot \left[ \rhot \uvt - \rho^-  \uv^- (1-\sigma)(1-\eta) \right] (1-\eta) + \nabla \uv^+ \cdot \left( \rhot \uvt - \rho^+ \uv^+ \sigma\eta \right) \eta \\
	& + \rhot\ut_1 \p_1 \eta(\zv^+ - \zv^-), \\
	R_7 =~& \rho^- \uv^- \dnab \uv^- (1-\sigma)\eta (\eta-2) + \rho^+ \uv^+\dnab \uv^+ \sigma (\eta^2-1) \\
	& + \left( \rhot\ut_1 - \rhot^r \ut_1^r \right) \p_1  \ut_1^r \E_1. 
	\end{align*}
	Moreover,
	\begin{align}
	\nabla p(\rhot) =~& p'(\rhot) \left[ \nabla \rho^- (1-\sigma) + \nabla \rho^+ \sigma + \p_1  \rhot^r \E_1 + \left(v^+ - v^-\right) \p_1  \sigma \E_1 \right] \notag \\
	=~& \left(p'(\rhot) - p'(\rho^-)\right) \nabla \rho^- (1-\sigma) + \nabla p(\rho^-) (1-\sigma) + \left(p'(\rhot) - p'(\rho^+)\right)\nabla \rho^+ \sigma \notag \\
	& + \nabla p(\rho^+) \sigma + \left(p'(\rhot) - p'(\rhot^r) \right) \p_1  \rhot^r \E_1 + \nabla p(\rhot^r) + p'(\rhot)  \left(v^+-v^-\right)\p_1  \sigma \E_1 \notag \\
	=~& \nabla p(\rho^-) (1-\sigma) + \nabla p(\rho^+) \sigma + \p_1  p(\rhot^r) \E_1 + R_8, \label{h1-3}
	\end{align}
	where the remainder $ R_8 $ satisfies that
	\begin{align*}
	R_8 =~& \left(\rho^+-\rho^-\right) \left[ a(\rho^-, \rhot) \nabla \rho^- - a(\rho^+,\rhot) \nabla \rho^+ \right] \sigma(1-\sigma) \\
	& + a(\rhot^r,\rhot) (\rhot-\rhot^r) \p_1  \rhot^r \E_1 + p'(\rhot) \left(v^+-v^-\right)\p_1  \sigma \E_1,
	\end{align*}
	where $ a(u,v):= \int_{0}^{1} p''(u+\theta(v-u))d\theta. $
	
	The remaining second-order terms satisfy that
	\begin{equation}\label{h1-4}
	\begin{aligned}
	\lap \uvt =~ & \lap \left( \uvt- \ut_1^r\E_1\right) + \p_1 ^2 \ut_1^r \E_1 \\
	=~ & \lap \left[ \zv^- (1-\sigma) + \zv^+ \sigma \right] + \lap \left[ (\zv^+ - \zv^-) (\eta-\sigma) \right] + \p_1 ^2 \ut_1^r \E_1 \\
	=~ & \lap \uv^- (1-\sigma) + \lap \uv^+ \sigma + \p_1 ^2 \ut_1^r \E_1 + R_9, \\
	\nabla \dv \uvt =~ & \nabla \dv \left[ \zv^- (1-\sigma) + \zv^+ \sigma \right] + \nabla \dv \left[ (\zv^+ - \zv^-) (\eta-\sigma) \right] + \p_1 ^2 \ut_1^r \E_1 \\
	=~ & \nabla \dv \uv^- (1-\sigma) + \nabla \dv \uv^+ \sigma + \p_1 ^2 \ut_1^r \E_1 + R_{10},
	\end{aligned}
	\end{equation}
	where 
	\begin{align*}
	R_9 =~& \lap \left[ (\zv^+ - \zv^-) (\eta-\sigma) \right] + (\zv^+ - \zv^-) \p_1 ^2 \sigma + 2 \p_1  (\zv^+ - \zv^-) \p_1  \sigma, \\
	R_{10} =~& \nabla \dv \left[ (\zv^+ - \zv^-) (\eta-\sigma) \right] + \nabla \left[ (z_1^+ - z_1^-) \p_1  \sigma \right] + \left(\dv \uv^+ - \dv \uv^-\right)\p_1 \sigma \E_1.
	\end{align*}
	Collecting \cref{h1-1,h1-2,h1-3,h1-4}, one can get that
	\begin{align*}
	\h =~& \left[ \rho^- \p_t \uv^- + \rho^- \uv^- \dnab \uv^- + \nabla p(\rho^-) - \mu \lap \uv^- - (\mu+\lambda) \nabla\dv \uv^- \right] (1-\sigma) \\
	& + \left[ \rho^+ \p_t \uv^+ + \rho^+ \uv^+ \dnab \uv^+ + \nabla p(\rho^+) - \mu \lap \uv^+ - (\mu+\lambda) \nabla\dv \uv^+ \right] \sigma \\
	& + \left[\rhot^r \p_t\ut_1^r + \rhot^r \ut_1^r \p_1 \ut_1^r + \p_1  p(\rhot^r) \right] \E_1  + h_0 \uvt \\
	& + \sum_{i=4}^8 R_i - \mu R_9 - (\mu+\lambda) R_{10} - (2\mu+\lambda) \p_1 ^2 \ut_1^r \E_1, \\
	=~& h_0 \uvt + \sum_{i=4}^8 R_i - \mu R_9 - (\mu+\lambda) R_{10} - (2\mu+\lambda) \p_1 ^2 \ut_1^r \E_1,
	\end{align*}
	which satisfies 
	\begin{equation*}
	\norm{\h+(2\mu+\lambda) \p_1 ^2 \ut_1^r \E_1}_{W^{1,p}(\Omega)} \leq C\norm{h_0}_{W^{1,p}(\Omega)} + C \sum_{i=4}^{10} \norm{R_i}_{W^{1,p}(\Omega)} \leq C \e e^{-\alpha t}.
	\end{equation*}
	
\end{proof}



\vspace{1.5cm}

\end{document}